\documentclass[twoside,12pt,leqno]{article}
\usepackage{amsmath, amsthm, amsfonts, amssymb, color}
\usepackage{mathrsfs}
\usepackage{fancyhdr}

\usepackage[active]{srcltx}
\usepackage{enumerate}
\usepackage{amsopn} 
\usepackage{bbm} 
\usepackage{mathrsfs}

\newcommand{\be}{\begin{eqnarray}}
\newcommand{\ee}{\end{eqnarray}}
\newcommand{\ce}{\begin{eqnarray*}}
\newcommand{\de}{\end{eqnarray*}}

\newtheorem{thm}{Theorem}[section]

\newtheorem{lem}[thm]{Lemma}

\newtheorem{exa}[thm]{Example}

\theoremstyle{definition}
\newtheorem{defn}{Definition}[section]

\definecolor{wco}{rgb}{0.5,0.2,0.3}

\numberwithin{equation}{section}
\theoremstyle{remark}
\newtheorem{rem}{Remark}[section]

\def\R{\mathbb{R}}

\def\<{\langle} \def\>{\rangle}  
    
\def\d{\text{\rm{d}}}   
  
 \def\beq{\begin{equation}}  \def\F{\mathcal F}

 \def\P{\mathbf P} 
 
 \def\ee{\varepsilon}

\def\L{\Lambda}
\def\W{W_0^{m,2}}
\def\H{H^m}
\def\P{\mathcal{P}}

\def\[{{\Big[}}
\def\]{{\Big]}}

\def\({{\Big(}}
\def\){{\Big)}}

\title{{\bf   Local and global well-posedness of SPDE with generalized coercivity conditions
}
}
\author{{\bf  Wei Liu $^a$\footnote{Corresponding author: weiliu@math.uni-bielefeld.de} ,
 Michael R\"{o}ckner $^{a,b}$
}\\
{\footnotesize  $a.$  Fakult\"at f\"ur Mathematik, Universit\"at Bielefeld,
D-33501 Bielefeld, Germany}\\
  \footnotesize{$b.$ Department of Mathematics and Statistics, Purdue University, West Lafayette, 47906 IN, USA}\\
}
\date{}
\begin{document}
\maketitle

\begin{abstract}
In this paper we establish   local and global  existence and uniqueness of solutions
for general  nonlinear  evolution equations with coefficients satisfying  some local monotonicity and generalized coercivity conditions.
 An analogous result is obtained  for  stochastic evolution equations in Hilbert  space with general additive noise.
As applications,  the main results  are  applied to obtain simpler proofs in known cases as the stochastic  3D Navier-Stokes equation,
 the tamed 3D Navier-Stokes equation and the Cahn-Hilliard equation, but also to get
 new results for stochastic surface growth PDE and stochastic power law fluids.
\end{abstract}

\bigbreak

\section{Main results}
Let  $(H, \<\cdot,\cdot\>_H)$ be a real separable
Hilbert space and identified with its dual space $H^*$ by the Riesz
isomorphism. Let   $V$ be  a real reflexive  Banach space such that it is
 continuously and densely embedded into $H$. Then we have the
 following Gelfand triple
$$V \subseteq H\equiv H^*\subseteq V^*.$$
   If  $\<\cdot,\cdot\>_V$ denotes  the
 dualization
between  $V$ and its dual space $V^*$,  then it is easy to show that
$$ \<u, v\>_V=\<u, v\>_H, \ \  u\in H ,v\in V.$$
Now we consider the general nonlinear evolution equation
\begin{equation}\label{1.1}
 u'(t)=A(t,u(t)),  \ 0<t<T,  \ u(0)=u_0\in H,
\end{equation}
where $T>0$, $u'$ is the generalized derivative of $u$ on $(0,T)$ and $A:[0,T]\times V\rightarrow V^*$ is restrictedly   measurable, $i.e.$  for  each $dt$-version of
$u\in L^1([0,T]; V)$, $t\mapsto A(t,u(t))$ is $V^*$-measurable on $[0,T]$.

A classical result says that (\ref{1.1}) has a unique solution if $A$
satisfies the monotonicity  and coercivity conditions (see e.g.
\cite{Ba10,Br73,Li69, Sh97,Z90} for more detailed exposition and references). The proof is mainly based on the Galerkin
approximation and the Minty  (monotonicity) trick. In \cite{L11}, the existence and uniqueness result was established
by replacing the monotonicity condition with a local version (see $(H2)$ below).
The result was applied to many new fundamental  equations
within this variational framework such as Burgers type equations, 2D Navier-Stokes equation and
the  3D Leray-$\alpha$ model. One of the main steps
in the proof in \cite{L11} was to show that any operator satisfying  local monotonicity is pseudo-monotone.
One should remark that the notion of a  pseudo-monotone operator
is  one of the most important extensions of the notion of a monotone operator and it was
 first introduced
 by Br\'{e}zis in \cite{Br68}.  The prototype of a pseudo-monotone operator is the sum of a monotone operator and
a strongly continuous operator ($i.e.$ an
 operator maps a weakly convergent sequence into a strongly convergent sequence). Hence
the theory of  pseudo-monotone operators unifies both  monotonicity arguments and
compactness arguments  (cf. \cite{Sh97,Z90}).

Also for stochastic partial differential equations (SPDE), the above approach, also called
the variational approach,  has  been  used extensively by many authors.
 The existence and uniqueness of solutions for SPDE was first investigated by Pardoux \cite{Par75}, Krylov and
Rozovskii \cite{KR79}. We refer e.g. to \cite{G,RRW07,Zh09} for some further
generalizations.  In particular,
   the local monotonicity condition has  been  used to establish  well posedness for SPDE in \cite{LR10,BLZ}.
 For further references on  various types of properties established for SPDE
within the variational framework, we refer to \cite{CM10,GLR,LR10,Zh09}.

 In this work we  establish  existence, uniqueness and continuous dependence on initial conditions
 of solutions to (\ref{1.1})
by using  the local monotonicity condition (see $(H2)$ below) and the
 generalized coercivity condition $(H3)$ defined below.  An analogous result for
stochastic PDE with general additive noise is  also obtained.  The standard
growth condition on $A$ (cf. \cite{Ba10,KR79,Li69,Z90}) is also replaced by a much weaker condition
 such that the main result
can be applied to a larger class of examples.  This result seems new even in the finite dimensional case.
The main result can be applied
 to establish the local/global  existence and uniqueness of solutions for a large class of
 classical (stochastic)  nonlinear evolution equations
such as the  stochastic 2D and 3D Navier-Stokes
equation, the tamed 3D Navier-Stokes equation and the Cahn-Hilliard equation. Through our generalized framework we give
 new and significantly simpler proofs for all these well known results.
 Moreover,  the main result is  also applied to
 stochastic surface growth PDE  and stochastic  power law fluids to obtain some new existence and uniqueness results for these models (see Section 3 for more details).
We emphasize that by applying the main result we obtain both the known local existence and uniqueness of
strong solutions to the  stochastic 3D Navier-Stokes equation and new local existence and uniqueness results for stochastic surface growth PDE.
Here the meaning of  strong solution is in the  sense of both PDE and stochastic analysis.

In particular, the (stochastic) 2D and 3D Navier-Stokes
equation are now included in this extended variational framework  using the local monotonicity and
generalized coercivity condition.
The study of stochastic Navier-Stokes equations dates back to the work of Bensoussan and Temam \cite{BT73}.
Although we have quite satisfactory results for 2D stochastic Navier-Stokes equations
 such as well-posedness, small noise asymptotics and  ergodicity  (cf.\cite{LR10,CM10,HM06} and the references therein),
 the results for the three dimensional case are still quite incomplete
 due to the lack of uniqueness (cf.\cite{BP99,CP97,DD03,DO06,FG95,FR08,MR04,MR05}).
Concerning the existence of solutions,
  in \cite{FG95} Flandoli and Gatarek proved the existence of martingale solutions and stationary
solutions for any dimensional stochastic Navier-Stokes equations in a bounded domain.
Subsequently, Mikulevicius and Rozovskii in \cite{MR05}
showed the existence of martingale solutions to stochastic Navier-Stokes equations
in $\mathbb{R}^d (d\ge 2)$ under weaker assumptions on the coefficients.

Replacing  the standard coercivity assumption (i.e.  $g(x)=Cx$ in  $(H3)$ below) by a more general version
 is  motivated by many reasons.  One  motivation is trying  to investigate the 3D Navier-Stokes equation
by applying our new result since we know that the local monotonicity hold for both the 2D and  3D Navier-Stokes equation.
However, as pointed out in \cite{L11,LR10},  the growth condition (see $(H4)$ below) fails to hold for the  3D Navier-Stokes equation.
On  the other hand, inspired by a series of works on the stochastic tamed 3D Navier-Stokes equation \cite{RZZ,RZ09a,RZ09},
we realized that, instead of working on the usual Gelfand triple $H^1\subseteq H^0 \subseteq H^{-1}$ (see Section 3 for details),
one may use the following Gelfand triple
 $$H^2 \subseteq H^1 \subseteq H^{0} .$$
On this triple one can verify the growth condition and also the local monotonicity for 3D Navier-Stokes equation, but
the usual coercivity condition does not hold anymore. Therefore,  we introduce the generalized coercivity condition $(H3)$  in order to overcome
this difficulty.  However, under this general form of   coercivity   one is only   able to get the
 local existence and uniqueness of solutions.  We should remark that our main result can  also be applied to the
 tamed 3D Navier-Stokes equation to get the global existence and uniqueness of solutions (see Section 3 for more examples).

Another reason of using this generalized coercivity condition  is coming from the proof of existence and uniqueness results for
stochastic evolution equations with general additive type noise. It is well known that
   stochastic equations (see (\ref{SDE}) below) can be reduced to deterministic evolution
 equations with a random parameter by a standard transformation (substitution).
 Then one can apply the result that we have already established for  deterministic
equations (cf. \cite{L11}).
However, $(H3)$  with the form of  $g(x)=Cx$ fails to hold in some examples  due to the more general growth condition $(H4)$
(see the proof of Theorem \ref{T1.1}).
But in such cases one will see that $(H3)$ still holds  with a certain non-decreasing continuous function $g$
 (e.g. $g(x)=Cx^\gamma$ for some  $C, \gamma>0$).  We refer to Section 3  for many examples  only satisfying this
generalized coercivity condition.

Now let us formulate the precise conditions on the coefficients in (\ref{1.1}).

  Suppose for  fixed $\alpha>1, \beta\ge 0$  there exist constants $\delta>0$,
  $C$ and
 a positive function $f\in L^1([0,T];  \mathbb{R})$
 such that the following conditions hold for all $t\in[0,T]$ and $v,v_1,v_2\in V$.
 \begin{enumerate}
 \item [$(H1)$] (Hemicontinuity)
      The map  $ s\mapsto \<A(t,v_1+s v_2),v\>_V$ is  continuous on $\mathbb{R}$.

\item[$(H2)$] (Local monotonicity)
     $$  \<A(t,v_1)-A(t, v_2), v_1-v_2\>_V
     \le \left(f(t)+\rho(v_1)+\eta(v_2) \right)  \|v_1-v_2\|_H^2, $$
where $\rho,\eta: V\rightarrow [0,+\infty)$ are measurable
and locally bounded functions on $V$.

  \item [$(H3)$] (Generalized coercivity)
    $$ 2 \<A(t,v), v\>_V  \le  -\delta
    \|v\|_V^{\alpha}  +g\left(\|v\|_H^2\right)+ f(t),$$
where  $g:[0,\infty) \rightarrow [0,\infty)$ is a non-decreasing continuous function.

  \item[$(H4)$] (Growth)
   $$ \|A(t,v)\|_{V^*} \le \bigg(  f(t)^{\frac{\alpha-1}{\alpha}} +
   C\|v\|_V^{\alpha-1} \bigg)  \bigg( 1+ \|v\|_H^{\beta} \bigg).$$

\end{enumerate}

\begin{rem}
(1) If $\rho=\eta\equiv 0$, $g(x)=Cx$ and $\beta=0$, then $(H1)$-$(H4)$ are the
classical monotonicity  and coercivity  conditions in \cite[Theorem
30.A]{Z90} (see also \cite{Ba10,KR79,Li69,PR07}).  It can be applied to
many quasilinear PDE such as porous medium equations and the $p$-Laplace
equation (cf.\cite{Z90,PR07}).

(2) If $f(t)\equiv C$ in $(H2)$ and $g(x)=Cx$ in $(H3)$,  existence and uniqueness  is obtained in \cite{L11} and the result
 is  applied to many examples such as  Burgers type equations,  the 2D Navier-Stokes equation, the 3D Leray-$\alpha$ model
 and the $p$-Laplace equation with non-monotone perturbations.
For readers  interested in stochastic partial differential equations we refer to
\cite{LR10,BLZ} where the existence and uniqueness of strong solutions  is established under another form of
local monotonicity condition (namely  $\rho\equiv0$).

(3) We remark that  $(H2)$ also covers other  non-Lipschitz conditions used in the literature (cf. e.g. \cite{FZ05}).
Moreover, with small modifications to the proof,  $(H3)$ can be replaced by the following slightly modified condition:
$$   2 \<A(t,v), v\>_V  \le  -\delta
    \|v\|_V^{\alpha}  +  h(t) g\left(\|v\|_H^2\right)+ f(t),$$
where $h:[0,T]\rightarrow[0, \infty)$ is an  integrable function.
\end{rem}

Now we can state the main result, which gives a more general framework to analyze various classes of nonlinear
evolution equations.

\begin{thm}\label{T1}
Suppose that $V \subseteq H$ is compact and $(H1)$-$(H4)$ hold.

(i) For any $u_0\in H $,
 there exists a constant $T_0\in (0, T]$ such that
 $(\ref{1.1})$ has a  solution  on $[0, T_0]$, i.e.
$$u\in L^\alpha([0,T_0];V)\cap C([0,T_0];H), \     u'\in  L^{\frac{\alpha}{\alpha-1}}([0,T_0];V^*)$$
and
$$ \<u(t), v\>_H=\<u_0,v\>_H + \int_0^t \<A(s,u(s)),  v\>_V d s , \ t\in[0,T_0], v\in V.  $$
Moreover,  if there exist nonnegative constants $C$ and $\gamma$ such that
\begin{equation}\label{c3}
  \rho(v)+\eta(v) \le C(1+\|v\|_V^\alpha) (1+\|v\|_H^\gamma), \
v\in V,
\end{equation}
 then the solution of $(\ref{1.1})$ is unique on $[0, T_0]$.

(ii) If  $(H3)$ holds with $g(x)=Cx$ for some constant $C$,  then  all assertions in (i) hold  on $[0, T]$ (i.e. $T_0=T$).
\end{thm}

\begin{rem}\label{R1}
(1) In the proof one can see that  $T_0$ is a constant depending on $u_0, f$ and $g$.
More precisely, one can take any constant  $T_0$ which  satisfies the following property:
$$  0<T_0\le T \ \ \text{and }\ \ T_0<\sup_{x\in(0,\infty)} G(x)- G\left(\|u_0\|_H^2+\int_0^{T_0} f(s) d s \right) , $$
where the function $G$ is the one defined in Lemma \ref{Bihari inequality}.

In particular, if $g(x)=c_0 x^\gamma (\gamma\ge 1)$, then one can take any $T_0\in(0, T]$ satisfying
$$T_0 < \frac{c_0}{\left(\gamma-1\right)\left(\|u_0\|_H^2 +\int_0^{T_0}f(s) d s \right)^{\gamma-1}  }.  $$

(2) If $\rho\equiv 0$ or $\eta\equiv 0$ in $(H2)$, then the compactness assumption of $V\subseteq H$ can be
removed by using a different  proof (cf. \cite{LR10}).
Therefore, the result can also be applied to many nonlinear evolution equations with unbounded underlying domains.

\end{rem}

The next result shows the continuous dependence of solution of $(\ref{1.1})$ on the initial condition  $u_0$.
\begin{thm}\label{T2}
Suppose that $V \subseteq H$ is compact and $(H1)$-$(H4)$ hold, and $u_i$ are  solutions of $(\ref{1.1})$ on $[0,T_0]$
for initial conditions $u_{i,0}\in H$, $i=1,2$ respectively and
 satisfying 
$$ \int_0^{T_0}\left( \rho(u_1(s))+\eta(u_2(s)) \right)  d s<\infty.  $$
Then there exists a constant $C$ such that
\begin{equation}
 \begin{split}
      \|u_1(t)-u_2(t)\|_H^2
\le   \|u_{1,0}-u_{2,0}\|_H^2   \exp\left[\int_0^t \left( f(s)+\rho(u_1(s))+\eta(u_2(s)) \right) d s  \right], \ t\in[0, T_0].
 \end{split}
\end{equation}
\end{thm}


Now we formulate the analogous result for SPDE  in Hilbert space with  additive type
noise. Suppose that $U$ is a Hilbert space and
$W(t)$  is  a $U$-valued cylindrical Wiener process  defined on a
filtered probability space $(\Omega,\mathcal{F},\mathcal{F}_t,\mathbb{P})$.
We consider the following type of stochastic evolution equations  on $H$,
\begin{equation}\label{SDE}
d X(t)=\left[ A_1(t, X(t))+ A_2(t,X(t)) \right] dt  + B(t) d W(t), \ 0< t< T, \  X(0)=X_0,
\end{equation}
where  $A_1,A_2:[0,T]\times V\rightarrow V^*$ and $B: [0,T]\rightarrow L_2(U; H)$ (here $(L_2(U; H), \|\cdot\|_2)$
denotes the space of all Hilbert-Schmidt operators from $U$ to $H$) are measurable.

Now we give the definition of a  local solution to (\ref{SDE}).  We use $\tau$ to denote a stopping time in
the  filtered probability space $(\Omega,\mathcal{F},\mathcal{F}_t,\mathbb{P})$.
\begin{defn}\label{def:soln} An $H$-valued  $\mathcal{F}_t$-adapted process $\{X(t)\}_{t\in [0,\tau]}$ is called a local  solution
of $(\ref{SDE})$ if $X(\cdot,\omega) \in L^1([0,\tau(\omega)]; V) \cap L^2([0,\tau(\omega)]; H)$ and   $\mathbb{P}$-$a.s.\ \omega \in \Omega$,
  $$X(t)=X_0+\int_0^t \left[ A_1(s, X (s)) + A_2(s,X (s)) \right] d s + \int_0^t B(s) d W(s), \ 0<t<\tau(\omega), $$
where $\tau$ is a stopping time satisfying $\tau(\omega)>0$, $\mathbb{P}$-$a.e.\  \omega \in \Omega $ and
$X_0\in L^2(\Omega\rightarrow H; \mathcal{F}_0; \mathbb{P})$.
\end{defn}

\begin{thm}\label{T1.1}
Suppose that $V \subseteq H$ is compact, $A_1$ satisfies $(H1)$-$(H4)$ with $\rho\equiv0$, $\beta=0$ and $g(x)=C x$,
$A_2$ satisfies $(H1)$-$(H4)$,
   $B \in L^2([0,T]; L_2(U; H))$, and
there exist nonnegative  constants $C$ and $\gamma$ such that
\begin{equation*}
\begin{split}
& \rho(u+v) \le C (\rho(u) + \rho(v)), \  u,v\in V; \\
& \eta(u+v) \le C( \eta(u) + \eta(v)  ),  \  u,v\in V; \\
&  \rho(v)+\eta(v) \le C(1+\|v\|_V^\alpha) (1+\|v\|_H^\gamma), \
v\in V.
\end{split}
\end{equation*}
Then  for any $X_0\in L^2(\Omega\rightarrow H; \mathcal{F}_0; \mathbb{P})$, there exists   a unique local solution  $\{X(t)\}_{t\in [0,\tau]}$ to
 $(\ref{SDE})$  satisfying
$$X(\cdot)\in L^\alpha([0, \tau];V)\cap C([0, \tau];H),\ \mathbb{P}\text{-}a.s..$$
Moreover, if $g(x)=Cx$ in $(H3)$ and $\alpha\beta\le 2$,  then  all assertions above hold for  $\tau\equiv T$.
\end{thm}

\begin{rem}\label{r1}
(1) The main idea of the proof is to use a transformation to reduce SPDE (\ref{SDE})
to a deterministic evolution
 equation (with  some random parameter)  which Theorem \ref{T1} can be applied to. More precisely,
 we consider the  process $Y$ which solves  the following SPDE:
 \begin{equation}\label{equation for Y}
  d Y(t)= A_1(t, Y(t)) dt  + B(t) d W(t), \ 0< t< T, \  Y(0)=0.
  \end{equation}
Since $A_1$ satisfies $(H1)$-$(H4)$ with $\rho\equiv0$ and $g(x)=C x$, then the existence and uniqueness of
$Y(t)$ follows from  Theorem 1.1 in \cite{LR10}. Let $u(t)=X(t)-Y(t)$, then it is easy to show that
$u(t)$ satisfies a deterministic evolution equation of type (\ref{1.1}) for each fixed $\omega\in \Omega$.

(2) Unlike in \cite{GLR}, here  we do not need to assume the noise to take values in $V$ (i.e. $B\in L_2(U; V)$).
The reason is that here we use the auxiliary process $Y$ instead of subtracting the noise part  directly as in \cite{GLR} and that $A_1\neq 0$ because it satisfies $(H3)$.

(3) One can replace the Wiener process  $W(t)$ in (\ref{SDE})  by  a L\'{e}vy type noise $L(t)$.
Then the existence and uniqueness of solutions to (\ref{equation for Y}) can be obtained from
the main result  in \cite{BLZ}, and the rest of the proof can be carried out similarly.

More generally, one might  replace   $W(t)$ in (\ref{SDE})  by a $U$-valued adapted stochastic process  $N(t)$
  with c\`{a}dl\`{a}g paths.
$N(t)$ can be various types of noises here. For instance,
 one can take $N(t)$ as  cylindrical Wiener process, fractional Brownian motion
or  L\'{e}vy process (cf.\cite{GLR}).  This subject and some further applications will be investigated
in future work.

(4) Comparing with the result obtained in \cite{LR10,BLZ}, the  Theorem above can be  applied  to
SPDE with more general drifts  (see Section 3 for many examples) provided the noise is of  additive type.
On the other hand,  the result in \cite{LR10,BLZ} is
 applicable  to SPDE with general  multiplicative  Wiener  noise or L\'{e}vy noise  if $\rho\equiv 0$ in $(H2)$ and $g(x)=Cx$ in $(H3)$.
\end{rem}


The rest of the paper is organized as follows. The proofs of the main results are given in the next section.
In Section 3 we apply the main results to several concrete (stochastic) semilinear and quasilinear evolution
equations  in Banach space. Throughout the paper,  we use $C$ to denote some   generic constant which might change from line to line.

\section{Proofs of The Main Theorems}

\subsection{Proof of Theorem \ref{T1}}

We will first consider the Galerkin approximation to (\ref{1.1}).  However, even in the finite dimensional case,
 the existence and uniqueness of solutions to (\ref{1.1}) seems not obvious because of the local monotonicity $(H2)$
 and the generalized coercivity condition $(H3)$. Here we prove it by using a
classical existence theorem of Carath\'{e}odory for ordinary differential equations.
Another difference  is that we can not  apply Gronwall's lemma directly for
this  general form of coercivity condition $(H3)$.  Instead, we will use  Bihari's inequality,
which is a  generalized
version of Gronwall's lemma (cf.\cite{B56,RZ}).
\begin{lem}\label{Bihari inequality}
(Bihari's inequality) Let $g:(0,\infty)\rightarrow (0,\infty)$ be a non-decreasing continuous function.
If $p$, $q$ are two positive functions on $\mathbb{R}^+$ and $K\ge 0$ is a constant such that
$$ p(t) \le K  +\int_0^t q(s)g\left(p(s) \right) d s,\ t\ge 0.  $$

(i) Then we have
\begin{equation}\label{Bihari}
   p(t)\le G^{-1}\left( G\left( K \right)+\int_0^t q(s) d s  \right),\  0\le t\le T_0,
\end{equation}
where $G(x):=\int_{x_0}^x \frac{1}{ g(s) } d s$ is well defined for some $x_0>0$, $G^{-1}$ is the inverse function and
$T_0\in(0, \infty)$ is a constant such that  $G( K )+\int_0^{T_0} q(s) d s$ belongs to the domain of $G^{-1}$.

(ii) If $K=0$ and there exists some $\varepsilon>0$ such that
$$      \int_0^\varepsilon  \frac{1}{g(s)}  d s=+\infty,   $$
then $p(t)\equiv 0 $.
\end{lem}

\begin{rem}\label{r2.0}
  It is obvious that the interval  $[G(K), \sup_{x\in(0,\infty)}G(x))$ is contained in the domain of $G^{-1}$, hence  (\ref{Bihari})
holds for $t\in[0, T_0]$,  where   $T_0$ satisfies
$$  \int_0^{T_0} q(s) ds <  \sup_{x\in(0,\infty)}G(x) -G(K).  $$
In particular, if $q\equiv 1$ and  $g(x)=C_0 x^\gamma$ for some constant $C_0>0$ and  $\gamma\ge 1$,  then
$$        G(x)=\frac{C_0}{\gamma-1} \left( x_0^{1-\gamma} - x^{1-\gamma}  \right);      ~
G^{-1}(x)= \left( x_0^{1-\gamma}  - \frac{\gamma-1}{C_0} x\right)^{ \frac{1}{1-\gamma} }.          $$
Hence (\ref{Bihari}) holds on  $[0, T_0]$ for  any $T_0\in [0, \frac{C_0}{\gamma -1}K^{1-\gamma})$ (in particular, for any  $T_0\in[0,\infty)$ if $\gamma=1$).


\end{rem}

Another difficulty  is due to the local monotonicity. It is well known that the hemicontinuity and
(global) monotonicity implies demicontinuity (cf. \cite{PR07,Z90}), which implies continuity in the finite dimensional case.
This is crucially used in the proof of existence of solutions for the finite dimensional equations of
the Galerkin approximation. In order to show the demicontinuity of locally monotone operators,
we need to use the techniques of
pseudo-monotone operators.
We first recall the definition of a  pseudo-monotone operator, which is  a very useful
generalization of a monotone operator and was first
introduced
 by Br\'{e}zis in \cite{Br68}.  We use the  notation ``$\rightharpoonup$''
for  weak convergence in Banach spaces.

\begin{defn} The operator $A: V\rightarrow V^*$ is called pseudo-monotone  if $v_n\rightharpoonup v$ in $V$ and
$$     \liminf_{n\rightarrow\infty} \<A(v_n), v_n-v\>_V\ge 0 $$
implies for all $u\in V$
$$   \<A(v), v-u\>_V \ge  \limsup_{n\rightarrow\infty} \<A(v_n), v_n-u\>_V.  $$
 \end{defn}

\begin{rem}\label{r2.1} Browder introduced
 a slightly different definition of  a pseudo-monotone operator in \cite{Bro77}:
An operator $A: V\rightarrow V^*$ is called pseudo-monotone  if
$v_n\rightharpoonup v$ in $V$ and
$$     \liminf_{n\rightarrow\infty} \<A(v_n), v_n-v\>_V\ge 0 $$
implies
$$   A(v_n)\rightharpoonup A(v)\ \    \text{and}   \  \
\lim_{n\rightarrow\infty} \<A(v_n), v_n\>_V=\<A(v), v\>_V.  $$
In particular, under
assumption  $(H4)$,  these two
definitions are equivalent (cf. \cite{L11}).
\end{rem}

\begin{lem}\label{L2.1}
 If the  embedding $V\subseteq H$ is compact, then $(H1)$ and $(H2)$  imply that
$A(t,\cdot)$ is pseudo-monotone for any $t\in[0,T]$.
\end{lem}
\begin{proof}  For the proof we refer to \cite[Lemma 2.5]{L11}.
\end{proof}

The proof of  Theorem  \ref{T1} is split into a few lemmas.
We first consider the Galerkin approximation to (\ref{1.1}).

 Let $\{e_1,e_2,\cdots \}\subset V$ be an orthonormal basis in $H$ and
 let $H_n:=span\{e_1,\cdots,e_n\}$ such that $span\{e_1,e_2,\cdots\}$ is dense in $V$. Let $P_n:V^*\rightarrow H_n$ be defined by
$$ P_ny:=\sum_{i=1}^n \<y,e_i\>_V e_i, \ y\in V^*.  $$
Obviously, $P_n|_H$ is just the orthogonal projection onto $H_n$ in H and we have
$$ \<P_nA(t,u), v\>_V=\<P_nA(t,u),v\>_H=\<A(t,u),v\>_V, \ u\in V, v\in H_n.  $$

For each finite $n\in \mathbb{N}$ we consider the following evolution equation on $H_n$:
\begin{equation}\label{2.1}
 u_n'(t)=P_nA(t,u_n(t)),  \  0<t<T, \ u_n(0)=P_nu_0\in H_n.
\end{equation}
From now on, we fix $T_0$ as a positive constant satisfying
$$  0<T_0\le T \ \ \text{and }\ \ T_0<\sup_{x\in(0,\infty)} G(x)- G\left(\|u_0\|_H^2+\int_0^{T_0} f(s) d s \right) , $$
where the function $f$ and $G$ are as in $(H3)$ and Lemma \ref{Bihari inequality} respectively.

In particular, if $g(x)=C_0 x^\gamma (\gamma\ge 1)$, then one can take any $T_0\in(0, T]$ satisfying
$$T_0 < \frac{C_0}{\left(\gamma-1\right)\left(\|u_0\|_H^2 +\int_0^{T_0}f(s) d s \right)^{\gamma-1}  }.  $$

\begin{lem}
  Under the assumptions of Theorem \ref{T1},    (\ref{2.1}) has a  solution on $[0,T_0]$. Moreover,
the solution is unique on $[0,T_0]$  if additionally (\ref{c3})  holds.
\end{lem}

\begin{proof} For any $t\in[0,T]$,
it is easy to show that $A(t,\cdot)$ is demicontinuous by $(H1)$ and $(H2)$ (cf. \cite[Remark 4.1.1]{PR07} or \cite[Proposition 26.4]{Z90}),  i.e.
$$  u_n\rightarrow u  ~\text{strongly in} ~V  ~\text{as}~ n\rightarrow\infty  $$
implies that
$$   A(t, u_n)\rightarrow A(t,u) ~\text{weakly in }~ V^* ~ \text{as} ~ n\rightarrow\infty.     $$
In fact, one can first show that $A$ is locally bounded by using  similar arguments as in \cite{PR07}. This implies that
$\{A(t,u_n)\}$ is bounded in $V^*$. Hence  there exists a subsequence $(n_k)_{k\in\mathbb{N}}$ and $w\in V^*$ such that
$A(t,u_{n_k})\rightarrow w$  weakly in $V^*$ as $k\rightarrow\infty$.

Since $u_{n_k}\rightarrow u$ strongly in $V$ as $k\rightarrow\infty$, we have
$$        \lim_{k\rightarrow\infty} \<A(t,u_{n_k}), u_{n_k}\>_V =\<w, u  \>_V.   $$
By Lemma \ref{L2.1} we know that $A(t,\cdot)$ is a pseudo-monotone operator. Then by Remark \ref{r2.1}
we can conclude that  $A(u)=w$.  Since for all such subsequences their weak limit is $A(u)$, we have
$$   A(t, u_n)\rightarrow A(t,u) ~\text{weakly in }~ V^*~ \text{as} ~ n\rightarrow\infty.     $$

In particular, the demicontinuity implies that  $P_nA(t, \cdot): H_n\rightarrow H_n$ is continuous and hence the functions
$$ (t,u)\rightarrow   \<P_nA(t,u), e_j\>_V, j=1,2,\cdots,n, $$
satisfy the Carath\'{e}odory condition on $[0,T]\times H_n$, i.e.  for all  $j=1,2,\cdots,n$
$$  t\rightarrow   \<P_nA(t,u), e_j\>_V ~\text{is measurable on }  ~ [0,T]~\text{for all}~  u\in H_n;  $$
$$  u\rightarrow   \<P_nA(t,u), e_j\>_V ~\text{is continuous on  }  ~ H_n~\text{for almost  all}~  t\in [0,T]. $$

By $(H3)$ and Lemma \ref{Bihari inequality} we  get the following a priori estimate for (\ref{2.1}) (see Lemma \ref{l2.3}):

 There exist positive constants $T_0$ and $c$ such that if $u: I_0\rightarrow H_n$ is a solution of (\ref{2.1}) on an
arbitrary subinterval $I_0$ of $[0,T_0]$, then
$$    \|u(t)\|_H \le c  ~\text{for~ all}~   t\in I_0.     $$

Therefore, according to the classical existence theorem of Carath\'{e}odory for ordinary differential equations in $\mathbb{R}^n$
 (cf. \cite[pp. 799-800]{Z90}), there exists a unique solution $u_n$ to (\ref{2.1}) on $[0,T_0]$ such that
 $$u_n\in L^\alpha([0,T_0];H_n)\cap C([0,T];H_n), \      u_n'\in  L^{\frac{\alpha}{\alpha-1}}([0,T_0];H_n) .$$
\end{proof}

\begin{rem} From the proof it is clear that
the constant $T_0$  comes from the application of Bihari's inequality.  It only depends on
$u_0,g,f$ and is independent of $n$.
\end{rem}

For the  constant $T_0\in(0, T]$, let
$X:=L^\alpha([0,T_0];V)$, then
$X^*=L^{\frac{\alpha}{\alpha-1}}([0,T_0];V^*)$. We denote by
$W^1_\alpha(0,T_0;V, H)$ the Banach space
$$ W^1_\alpha(0,T_0;V, H)=\{ u\in X: u'\in X^*  \}, $$
where $u^\prime$ is the weak derivative of
$$t \mapsto u(t)\in V\subseteq H \subseteq V^*$$
and on $W^1_\alpha(0,T_0;V, H)$ the norm is defined by
$$ \|u\|_W:=\|u\|_X+\|u'\|_{X^*}= \left(\int_0^{T_0}\|u(t)\|_V^\alpha dt\right)^{\frac{1}{\alpha}}+
\left(\int_0^{T_0}\|u'(t)\|_{V^*}^{\frac{\alpha}{\alpha-1}} dt\right)^{\frac{\alpha-1}{\alpha}}. $$
 It's well known that $W^1_\alpha(0,T_0;V, H)$ is a reflexive
Banach space and it is continuously embedded into $C([0,T_0];H)$ (cf. \cite{Z90}). Moreover, we also have the following
integration by parts formula
\begin{equation*}
 \begin{split}
\<u(t), v(t)\>_H -\<u(0), v(0)\>_H=\int_0^t& \<u'(s), v(s)\>_V ds
+ \int_0^t \<v'(s), u(s)\>_V ds,\\
& \ t\in[0,T_0],  \ u,v\in W^1_\alpha(0,T_0; V,H).
 \end{split}
\end{equation*}

\begin{lem}\label{l2.3}
Under the assumptions of Theorem \ref{T1}, we have for any solution $u_n$ to $(\ref{2.1})$
\begin{equation}\label{e9}
\begin{split}
   \|u_n(t)\|_H^2+ \delta \int_0^t\|u_n(s)\|_V^\alpha d s
 \le
G^{-1} \left( G\left( \|u_0\|_H^2
 +\int_0^{T_0}  f(s)  d s \right)
+ t \right),  \ t\in[0, T_0],
\end{split}
\end{equation}
where $G(x):=\int_{x_0}^x \frac{1}{ g(r) } d r$ is well defined for some $x_0>0$.

In particular, there exists a  constant    $K>0$ such that
\begin{equation}
 \|u_n\|_X+\sup_{t\in[0,T_0]}\|u_n(t)\|_H+\|A(\cdot,u_n)\|_{X^*}\le K, \ n\ge 1.
\end{equation}
\end{lem}

\begin{proof} By the integration by parts formula and $(H3)$ we have
\begin{equation}\begin{split}
& ~~~~\|u_n(t)\|_H^2-\|u_n(0)\|_H^2\\
&=2\int_0^t\<u_n'(s),u_n(s)\>_V d s\\
&=2\int_0^t\<P_nA(s,u_n(s)), u_n(s)\>_V d s\\
&=2\int_0^t\<A(s,u_n(s)), u_n(s)\>_V d s\\
&\le\int_0^t\left(-\delta\|u_n(s)\|_V^\alpha+ g\left( \|u_n(s)\|_H^2 \right) +f(s)  \right) d s.
\end{split}
\end{equation}
Hence we have for $t\in [0,T_0]$,
$$  \|u_n(t)\|_H^2+ \delta \int_0^t\|u_n(s)\|_V^\alpha d s\le \|u_0\|_H^2
 +\int_0^{T_0}   f(s)  d s
+ \int_0^t g\left( \|u_n(s)\|_H^2 \right) ds.  $$
Then by  Lemma \ref{Bihari inequality} and Remark \ref{r2.0}  we know that (\ref{e9}) holds.

Therefore, there exists a constant $C_2$ such that
$$ \|u_n\|_{X}+\sup_{t\in[0,T_0]}\|u_n(t)\|_H \le C_2, \ n\ge 1.  $$
Then by $(H4)$ there exists a constant $C_3$ such that
$$ \|A(\cdot, u_n)\|_{X^*} \le C_3, \ n\ge 1.   $$
Hence the proof is complete.
\end{proof}

Note that $X,X^*$ and $H$ are reflexive spaces. Then by  Lemma \ref{l2.3}  there exists a subsequence, again denoted
by $u_n$, such that as $n\rightarrow\infty$
\begin{equation*}
 \begin{split}
    u_n &\rightharpoonup u\  \ \text{in}\  X \ \text{and} \ \   W^1_\alpha(0,T_0;V,H); \\
 A(\cdot,u_n)&\rightharpoonup w \ \ \text{in} \ X^*; \\
 u_n(T_0) &\rightharpoonup z \ \ \text{in}  \  H.
 \end{split}
\end{equation*}

Recall that $u_n(0)=P_nu_0\rightarrow u_0$ in $H$ as $n\rightarrow\infty$.
\begin{lem}
Under the assumptions of Theorem \ref{T1},
 the limit elements $u,w$ and $z$ satisfy  $u\in W^1_\alpha(0,T_0;V,H)$ and
$$ u'(t)=w(t),  \ 0<t<T_0, \ u(0)=u_0, \ u(T_0)=z.$$
\end{lem}
\begin{proof} See \cite[Lemma 2.3]{L11}.
\end{proof}

The next crucial step in the proof of Theorem \ref{T1} is to verify $w=A(u)$.
In the case of monotone operators, this is the well known Minty's lemma (or monotonicity trick) (cf. \cite{Mi62,Mi63,Bro63,Bro64}).
In the case of locally monotone operators, we use the following  integrated version of Minty's lemma
which holds due to pseudo-monotonicity.
The following lemma has  first been proved  in \cite[Lemma 2.6]{L11}.  We include the proof
 here for the reader's convenience.


\begin{lem}\label{L2.5}
Under the assumptions of Theorem \ref{T1}, and
supposing that
\begin{equation}\label{2.3}
 \liminf_{n\rightarrow \infty} \int_0^{T_0}\<A(t,u_n(t)), u_n(t)\>_V d t \ge
\int_0^{T_0} \<w(t), u(t)\>_V d t,
\end{equation}
 we have for any $v\in X$
\begin{equation}
\int_0^{T_0}\<A(t,u(t)), u(t)-v(t)\>_V d t \ge
 \limsup_{n\rightarrow \infty}\int_0^{T_0} \<A(t,u_n(t)), u_n(t)-v(t)\>_V d t.
\end{equation}
In particular, we have $A(t,u(t))=w(t), ~a.e.\  t\in[0,T_0]$.
\end{lem}

\begin{proof}
Since $W^1_\alpha(0,T_0;V,H)\subset C([0,T_0];H)$ is a continuous embedding, we have that
 $u_n(t)$ converges to $u(t)$ weakly in $H$
for all $t\in[0,T_0]$.

\textbf{Claim 1:}  For all  $t\in[0,T_0]$ we have
\begin{equation}\label{claim 1}
  \limsup_{n\rightarrow\infty}\<A(t,u_n(t)),u_n(t)-u(t)\>_V \le 0 .
\end{equation}

Suppose there exists a $t_0\in[0,T_0]$ such that
$$   \limsup_{n\rightarrow\infty}\<A(t_0,u_n(t_0)),u_n(t_0)-u(t_0)\>_V > 0.   $$
Then we can take a subsequence such that
$$   \lim_{i\rightarrow\infty}\<A(t_0,u_{n_i}(t_0)),u_{n_i}(t_0)-u(t_0)\>_V > 0.   $$

By $(H3)$ and $(H4)$ there exists a constant $K$ such that
\begin{equation*}
 \begin{split}
  2 \<A(t_0,u_{n_i}(t_0)),u_{n_i}(t_0)-u(t_0)\>_V  \le & -\frac{\delta}{2}\|u_{n_i}(t_0)\|_V^\alpha +K\left( f(t)
+ g\left( \|u_{n_i}(t_0)\|_H^2 \right)  \right)\\
& +K\left(1+\|u_{n_i}(t_0)\|_H^{\alpha \beta} \right)\|u(t_0)\|_V^\alpha.
 \end{split}
\end{equation*}
Hence we know that $\{u_{n_i}(t_0)\}$ is bounded in $V$ (w.r.t. $\|\cdot\|_V$), so there
exists a subsequence of $\{u_{n_i}(t_0)\}$  converges to some limit weakly in $V$.

Note that $u_{n_i}(t_0)$ converges to $u(t_0)$ weakly in $H$,
it is easy to show that  $u_{n_i}(t_0)$  converges to $u(t_0)$ weakly in $V$.

  Since $A(t_0,\cdot)$ is pseudo-monotone, we have
$$ \<A(t_0,u(t_0)), u(t_0)-v\>_V \ge \limsup_{i\rightarrow\infty}\<A(t_0,u_{n_i}(t_0)),u_{n_i}(t_0)-v\>_V,
\ v\in V. $$
In particular, we have
 $$ \limsup_{i\rightarrow\infty}\<A(t_0,u_{n_i}(t_0)),u_{n_i}(t_0)-u(t_0)\>_V\le 0, $$
which is a contradiction to the definition of the subsequence $\{ u_{n_i}(t_0)\}$.

Hence (\ref{claim 1}) holds.

Similarly, by $(H3)$ and $(H4)$ there exists a constant $K$ such that
\begin{equation*}
 \begin{split}
  2 \<A(t,u_n(t)),u_n(t)-v(t)\>_V\le & -\frac{\delta}{2}\|u_n(t)\|_V^\alpha +K\left( f(t)
+ g\left( \|u_n(t)\|_H^2 \right)  \right)\\
& +K\left(1+\|u_n(t)\|_H^{\alpha \beta} \right)\|v(t)\|_V^\alpha,\
v\in X.
 \end{split}
\end{equation*}

Then by Lemma \ref{l2.3}, Fatou's lemma, (\ref{2.3}) and (\ref{claim 1}) we have
\begin{equation}
\begin{split}
0&\le \liminf_{n\rightarrow \infty}\int_0^{T_0} \<A(t,u_n(t)), u_n(t)-u(t)\>_V d t \\
&\le \limsup_{n\rightarrow \infty}\int_0^{T_0} \<A(t,u_n(t)), u_n(t)-u(t)\>_V d t \\
&\le \int_0^{T_0} \limsup_{n\rightarrow \infty} \<A(t,u_n(t)), u_n(t)-u(t)\>_V d t \le 0.
\end{split}
\end{equation}

Hence
$$   \lim_{n\rightarrow \infty}\int_0^{T_0} \<A(t,u_n(t)), u_n(t)-u(t)\>_V d t = 0.   $$

\textbf{Claim 2:} There exists a subsequence $\{u_{n_i}\}$ such that
\begin{equation}\label{claim 2}
 \lim_{i\rightarrow \infty} \<A(t,u_{n_i}(t)), u_{n_i}(t)-u(t)\>_V  = 0 \  \ \text{for}\ a.e. \ t\in[0,T_0].
\end{equation}
Define $ g_n(t):= \<A(t,u_{n}(t)), u_{n}(t)-u(t)\>_V, \ t\in[0,T] $.  Then
$$     \lim_{n\rightarrow \infty}\int_0^{T_0} g_n(t) d t = 0, \ \ \ \limsup_{n\rightarrow\infty} g_n(t)\le 0,\   \  \  t\in[0,T_0].       $$
Then by Lebesgue's dominated convergence theorem we have
$$ \lim_{n\rightarrow \infty}\int_0^{T_0} g_n^+(t) d t = 0,  $$
where $g_n^+(t):=max\{g_n(t), 0 \}$.

Note that $|g_n(t)|=2g_n^+(t)-g_n(t)$, hence we have
$$  \lim_{n\rightarrow \infty}\int_0^{T_0} |g_n(t)| d t = 0.     $$
Therefore, we can take a subsequence $\{g_{n_i}(t)\}$ such that
$$     \lim_{i\rightarrow \infty} g_{n_i}(t) = 0 \ \ \text{for}\ a.e. \ t\in[0,T_0],  $$
$i.e.$  (\ref{claim 2}) holds.

Therefore, for any $v\in X$, we can choose a subsequence $\{u_{n_i}\}$ such that
$$   \lim_{i\rightarrow \infty}\int_0^{T_0} \<A(t,u_{n_i}(t)), u_{n_i}(t)-v(t)\>_V d t =
\limsup_{n\rightarrow \infty}\int_0^{T_0} \<A(t,u_n(t)), u_n(t)-v(t)\>_V d t;         $$
$$    \lim_{i\rightarrow \infty} \<A(t,u_{n_i}(t)), u_{n_i}(t)-u(t)\>_V  = 0 \ \ \text{for}\ a.e. \ t\in[0,T_0].       $$
Since $A$ is pseudo-monotone, we have
$$  \<A(t,u(t)), u(t)-v(t)\>_V \ge  \limsup_{i\rightarrow \infty} \<A(t,u_{n_i}(t)), u_{n_i}(t)-v(t)\>_V,\  \ t\in[0,T_0].       $$
By Fatou's lemma we obtain
\begin{equation}
 \begin{split}
  \int_0^{T_0}\<A(t,u(t)), u(t)-v(t)\>_V d t &\ge
 \int_0^{T_0}\limsup_{i\rightarrow \infty} \<A(t,u_{n_i}(t)), u_{n_i}(t)-v(t)\>_V d t\\
&\ge  \limsup_{i\rightarrow \infty}\int_0^{T_0} \<A(t,u_{n_i}(t)), u_{n_i}(t)-v(t)\>_V d t\\
&= \limsup_{n\rightarrow \infty}\int_0^{T_0} \<A(t,u_n(t)), u_n(t)-v(t)\>_V d t.
\end{split}
\end{equation}

In particular,  we have for any $v\in X$,
\begin{equation*}
 \begin{split}
  \int_0^{T_0}\<A(t,u(t)), u(t)-v(t)\>_V d t
&\ge\limsup_{n\rightarrow \infty}\int_0^{T_0} \<A(t,u_n(t)), u_n(t)-v(t)\>_V d t\\
&\ge\liminf_{n\rightarrow \infty}\int_0^{T_0} \<A(t,u_n(t)), u_n(t)-v(t)\>_V d t\\
&\ge \int_0^{T_0} \<w(t), u(t)\>_V d t- \int_0^{T_0} \<w(t), v(t)\>_V d t\\
&=\int_0^{T_0} \<w(t), u(t)-v(t)\>_V d t.
\end{split}
\end{equation*}
Since $v\in X$ is arbitrary, we have $A(\cdot,u)=w$ as  elements in $X^*$.

Hence the proof is complete.
\end{proof}

Now we can give the complete proof of Theorem \ref{T1}.

\noindent\textbf{Proof of Theorem \ref{T1}}
(i) Existence:
The integration by parts formula implies that
$$ \|u_n(T_0)\|_H^2-\|u_n(0)\|_H^2=2 \int_0^{T_0} \<A(t,u_n(t)),  u_n(t)\>_V d t;   $$
$$ \|u(T_0)\|_H^2-\|u(0)\|_H^2=2 \int_0^{T_0} \<w(t), u(t)\>_V d t.   $$
Since $u_n(T_0)\rightharpoonup z$ in $H$, by the lower semicontinuity of $\|\cdot\|_H$ we have
$$  \liminf_{n\rightarrow \infty} \|u_n(T_0)\|_H^2\ge \|z\|_H^2=\|u(T_0)\|_H^2.   $$
Hence we have
\begin{equation*}
 \begin{split}
&  \liminf_{n\rightarrow \infty} \int_0^{T_0}\<A(t,u_n(t)), u_n(t)\>_V d t \\
\ge & \frac{1}{2}
\left( \|u(T_0)\|_H^2 -\|u(0)\|_H^2 \right)\\
=& \int_0^{T_0} \<w(t), u(t)\>_V d t.
\end{split}
\end{equation*}
By Lemma \ref{L2.5}  we know that
 $u$ is a solution to (\ref{1.1}).

(ii) Uniqueness:  Suppose $u(\cdot,u_0),v(\cdot,v_0)$ are the
solutions to (\ref{1.1}) with starting points $u_0,v_0$
respectively, then by the integration by parts formula   we have for
$t\in[0, T_0]$,
\begin{equation*}
 \begin{split}
  \|u(t)-v(t)\|_H^2&=\|u_0-v_0\|_H^2+ 2\int_0^t\< A(s,u(s))-A(s,v(s)),u(s)-v(s)\>_V ds\\
 &\le \|u_0-v_0\|_H^2+ 2\int_0^t \left( f(s)+\rho(u(s))+ \eta(v(s))    \right)   \|u(s)-v(s)\|_H^2 ds.
 \end{split}
\end{equation*}
By (\ref{c3}) we know that
$$       \int_0^{T_0} \left( f(s)+\rho(u(s))+ \eta(v(s))  \right)  d s <\infty.                $$
Then by Gronwall's lemma we obtain
\begin{equation}\label{estimate of difference}
    \|u(t)-v(t)\|_H^2\le\|u_0-v_0\|_H^2 \exp\left[2\int_0^t \left( f(s)+\rho(u(s))+ \eta(v(s))    \right) d s  \right], \ t\in[0, T_0].
\end{equation}
In particular, if $u_0=v_0$, this  implies the  uniqueness of the
solution to $(\ref{1.1})$. \qed


\subsection{Proof of Theorem \ref{T2}} By $(H2)$ we have for $ t\in[0, T_0]$,
\begin{equation*}
 \begin{split}
&  \|u_1(t)-u_2(t)\|_H^2\\
=& \|u_{1,0}-u_{2,0}\|_H^2+ 2\int_0^t\< A(s,u_1(s))-A(s,u_2(s)),u_1(s)-u_2(s)\>_V ds\\
 \le&  \|u_{1,0}-u_{2,0}\|_H^2
    +\int_0^t \left( f(s)+\rho(u_1(s))+ \eta(u_2(s))   \right) \|u_1(s)-u_2(s)\|_H^2 ds,
 \end{split}
\end{equation*}
where $C$ is a constant.

Then by Gronwall's lemma we have
\begin{equation*}
 \begin{split}
      \|u_1(t)-u_2(t)\|_H^2
\le &  \exp\left[\int_0^t \left( f(s) +\rho(u_1(s))+ \eta(u_2(s))\right) d s  \right]\\
& \cdot  \left( \|u_{1,0}-u_{2,0}\|_H^2 +\int_0^t
\|b_1(s)-b_2(s)\|_H^2 ds \right), \ t\in[0, T_0].
 \end{split}
\end{equation*}
\qed

\subsection{Proof of Theorem \ref{T1.1}}
 We first  consider  the  process $Y$ which solves  the following SPDE:
 \begin{equation*}
  d Y(t)= A_1(t, Y(t)) dt  + B(t) d W(t), \ 0< t< T, \  Y(0)=0.
  \end{equation*}
By  \cite[Theorem 1.1]{LR10}  we know that there exists a unique solution $Y$ to the above equation and it  satisfies
$$  Y(\cdot)\in L^\alpha([0,T]; V)\cap C([0,T]; H); \  \mathbb{P}\text{-a.s.}.   $$
Let $u(t)=X(t)-Y(t)$.  Then it is easy to see that
$u(t)$ satisfies the following equation:
\begin{equation}\label{rde}
   u^\prime(t)= \tilde{A}\left(t, u(t) \right),\  0<t<T,
 \ u(0)=u_0,
\end{equation}
where
(for fixed $\omega$  which we omit  in the notation for simplicity)
  $$\tilde{A}(t,v) := A_1(t,v+Y(t))-A_1(t,Y(t))+A_2(t,v+Y(t)), \  v\in V.$$
It is easy to show that $\tilde{A}$ is a well defined operator from $[0,T]\times V$ to $V^*$ since  $Y(\cdot)\in L^\alpha([0,T]; V)$.

To obtain the  existence and uniqueness of  solutions to \eqref{rde}  we only need to show that $\tilde{A}$ satisfies all
 the assumptions of Theorem \ref{T1}.

Since $Y(t)$ is measurable,
$\tilde{A}(t,v)$ is $\mathcal{B}([0,T])\otimes \mathcal{B}(V)$-measurable.
It is also easy to show that
 $\tilde{A}$ is hemicontinuous
since  $(H1)$ holds for both $A_1$  and $A_2$.

For $u,v\in V$ we have
\begin{equation*}\begin{split}
& \<\tilde{A}(t,u)- \tilde{A}(t,v), u- v\>_V \\
=&  \<A_1(t,u+Y(t)) - A_1(t,v+Y(t)), u- v\>_V \\
& +  \<A_2(t,u+Y(t)) - A_2(t,v+Y(t)), u- v\>_V \\
\le & \left( f(t)+\eta(v+Y(t)) \right)\|u-v\|_H^2  \\
& + \left( f(t)+ \rho(v+Y(t)) +\eta(v+Y(t)) \right)\|u-v\|_H^2 \\
\le & C \left[ f(t)+ \rho(Y(t)) +\eta(Y(t)) + \rho(v) +\eta(v) \right] \|u-v\|_H^2,
\end{split}
\end{equation*}
i.e. $(H2)$ holds for $\tilde{A}$ with
$$  \tilde{f}(t)=C  \left[ f(t)+ \rho(Y(t)) +\eta(Y(t)) \right]\in L^1([0,T]).       $$
Since $A_2$ satisfies  $(H3)$ and $(H4)$,   by  Young's inequality we have
\begin{equation*}\begin{split}\label{coercivity}
&2 \<A_2(t,v+Y(t)), v\>_V=2 \<A_2(t,v+Y(t)), v+Y(t)-Y(t)\>_V\\
\le & -\delta \|v+Y(t) \|_V^\alpha + g\left( \|v+Y(t)  \|_H^2\right) +f(t)- 2 \<A_2(t,v+Y(t)), Y(t)\>_V\\
\le & -\delta \|v+Y(t) \|_V^\alpha +g\left( \|v+Y(t)  \|_H^2\right) +f(t) \\
& +C \left(f(t)^{\frac{\alpha-1}{\alpha}}+\|v+ Y(t)\|_V^{\alpha-1}\right)
\left( 1+\|v+Y(t)\|_H^\beta  \right)   \|Y(t)\|_V\\
\le & -\frac{\delta}{2} \|v+ Y(t)  \|_V^\alpha +g\left( \|v+Y(t)  \|_H^2\right)+ (1+\frac{\delta}{2}) f(t)\\
& + C \|Y(t)\|_V^\alpha \left(1+\|v+Y(t)\|_H^{\alpha\beta}\right)  \\
\le & -\frac{\delta}{2} \left(2^{1-\alpha}\|v\|_V^\alpha -  \| Y(t)  \|_V^\alpha \right)+  g\left( 2\|v \|_H^2+ 2  \|Y(t) \|_H^2\right)   \\
&   + (1+\frac{\delta}{2}) f(t) + C \|Y(t)\|_V^{\alpha} \left(1+\|v\|_H^{\alpha\beta} +  \|Y(t)\|_H^{\alpha\beta} \right)   \\
\le & -2^{-\alpha}\delta\|v\|_V^\alpha  +  g\left( 2\|v \|_H^2+ 2  \|Y(t) \|_H^2\right)  +
C \|Y(t)\|_V^{\alpha} \|v\|_H^{\alpha\beta}\\
 & + (1+\frac{\delta}{2}) f(t)
 + C \|Y(t)\|_V^{\alpha} \left(1+  \|Y(t)\|_H^{\alpha\beta} \right),  \ v\in V,
\end{split}
\end{equation*}
where $C$ is some constant changing from line to line (but independent of $t$ and $\omega$).

Similarly, we have
\begin{equation*}\begin{split}
&2 \<A_1(t,v+Y(t))-A_1(t,Y(t)), v\>_V \\
=&2 \<A_1(t,v+Y(t)), v+Y(t)-Y(t)\>_V-2  \<A_1(t,Y(t)), v\>_V \\
\le & -\delta \|v+Y(t) \|_V^\alpha + C \|v+Y(t)  \|_H^2 +f(t) \\
  &+ \|Y(t)\|_V\left( f(t)^{\frac{\alpha-1}{\alpha}}+ C\|v+Y(t)\|_V^{\alpha-1}    \right)+\|v\|_V\|A_1(t,Y(t))\|_{V^*}    \\
\le & -\frac{\delta}{2} \|v+ Y(t)  \|_V^\alpha +C \|v+Y(t)  \|_H^2 + (1+\frac{\delta}{2})f(t)  \\
& +C\|Y(t)\|_V^\alpha+\|v\|_V\|A_1(t,Y(t))\|_{V^*}    \\
\le & -\frac{\delta}{2} \left(2^{1-\alpha}\|v\|_V^\alpha -  \| Y(t)  \|_V^\alpha \right)+  C\left(\|v \|_H^2+  \|Y(t) \|_H^2\right)   \\
& + C(f(t)+\|Y(t)\|_V^\alpha)+\|v\|_V\left(    f(t)^{\frac{\alpha-1}{\alpha}}+ C\|Y(t)\|_V^{\alpha-1}   \right)    \\
\le & -2^{-\alpha-1}\delta\|v\|_V^\alpha  +  C\|v \|_H^2  + C \left(f(t)+ \|Y(t)\|_V^{\alpha} +  \|Y(t)\|_H^{2}  \right), \ v\in V.
\end{split}
\end{equation*}
Since $Y(\cdot)\in L^\alpha([0,T]; V)\cap C([0,T]; H)$, we know that $\tilde{A}$ satisfies $(H3)$ with
$$  \tilde{f}(t)=  C \left(f(t)+ \|Y(t)\|_V^{\alpha} +  \|Y(t)\|_H^{2} +\|Y(t)\|_V^\alpha \|Y(t)\|_H^{\alpha\beta}  \right).     $$
 The growth condition $(H4)$  also holds for $\tilde{A}$ since
\begin{align*}
  \|\tilde{A}(t, v)\|_{V^*}
  =& \|A_1(t,v+Y(t))\|_{V^*} + \|A_1(t,Y(t))\|_{V^*}+ \|A_2(t,v+Y(t))\|_{V^*}  \\
\le& C \left(f(t)^{\frac{\alpha-1}{\alpha}}+ \|v+Y(t) \|_V^{\alpha-1} \right) \left(1+ \|v+Y(t) \|_H^{\beta}  \right) \\
& +f(t)^{\frac{\alpha-1}{\alpha}}+ C\|Y(t) \|_V^{\alpha-1} \\
\le & \left(C f(t)^{\frac{\alpha-1}{\alpha}}+C\|Y(t)\|_V^{\alpha-1} +C \|v\|_V^{\alpha-1}  \right) \left(1+\|Y(t) \|_H^{\beta}+  \|v \|_H^{\beta}   \right ) \\
  \le& \left( \tilde{f}(t)^{\frac{\alpha-1}{\alpha}}+ C\|v\|_V^{\alpha-1} \right) \left(1+ \|v \|_H^{\beta}  \right).
\end{align*}

Therefore, according to Theorem \ref{1.1}, (\ref{rde})  has a unique local solution on $[0, T_0(\omega)]$ for $\mathbb{P}$-$a.s. \omega$.



Define
$$X(t):=u(t)+Y(t), $$
 then it is easy to show that $X(t)$ is the unique local solution to (\ref{SDE}).

Now the proof is complete.
\qed

\section{Application to Examples}

Since Theorem \ref{T1} is a generalization of a classical result for  monotone operators (cf.\cite{Ba10,Li69,Sh97,Z90})
and of  a recent result for locally monotone operators (cf.\cite{L11,LR10}), it can be applied to a large class
of semilinear and quasilinear evolution equations such as reaction-diffusion equations, generalized Burgers equations,
2D Navier-Stokes equation, 2D magneto-hydrodynamic equations, 2D magnetic B\'{e}nard problem, 3D Leray-$\alpha$
model, porous medium equations and generalized $p$-Laplace equations with locally monotone perturbations (cf.\cite{CM10,L11,LR10,PR07}).
In this section we will first  apply  our general results  to some known cases (Subsection 3.1, 3.2 and 3.3), but which  have not been covered by the more restricted framework in the above  references.
Subsequently, in Subsections 3.4 and 3.5 we apply our results to cases, which are not covered in the
existing literature, at least not in such generality.

\subsection{3D Navier-Stokes equation}
As we mentioned in the introduction,  the first  example here is  to apply Theorem \ref{T1} to the 3D Navier-Stokes equation, which is a
 classical model to  describe the time evolution of an incompressible fluid, given as follows:
\begin{equation}\label{3D NSE}
\begin{split}
    \partial_t u(t)&=\nu \Delta u(t)-(u(t)\cdot \nabla)u(t)+\nabla p(t)+f(t), \\
   \text{div} (u)&=0,   ~ u(0)=u_0,
\end{split}
\end{equation}
where $u(t,x)=(u^1(t,x), u^2(t,x), u^3(t,x))$ represents the velocity field of the fluid, $\nu$ is the viscosity constant,
the pressure $p(t,x)$
is an unknown scalar function and $f$ is  a  (known) external force field acting on the fluid. In the pioneering work
\cite{L34} Leray proved the existence of
a weak solution for the 3D Navier-Stokes equation in the whole space. However, up to now,
 the uniqueness and regularity of  weak solutions are still open problems (cf.\cite{Li96,T84}).

Let $\Lambda$ be a smooth bounded open domain in $\mathbb{R}^3$.
Let $C_0^\infty(\Lambda, \R^3)$ denote the set of all smooth functions from $\L$ to $\R^3$
with compact support.  For $p\ge 1$, let $L^p:=L^p(\L, \R^3)$ be the vector valued $L^p$-space in which the norm
is denoted by $\|\cdot\|_{L^p}$. For any integer $m\ge 0$, let $W_0^{m,2}$ be the standard Sobolev space on $\L$
with values in $\R^3$, i.e. the closure of $C_0^\infty(\Lambda, \R^3)$ with respect to the norm:
$$ \|u\|_{\W}^2 =\int_{\L} |(I-\Delta)^{\frac{m}{2}}  u|^2  \d x.  $$

For the reader's convenience, we recall the following Gagliardo-Nirenberg interpolation inequality,
 which plays an essential role
in the study of Navier-Stokes equations.

If $q\in [1,\infty]$ such that
$$    \frac{1}{q}= \frac{1}{2} - \frac{m\alpha}{3},\ 0\le \alpha\le 1,  $$
then there exists a constant $C_{m,q}>0$ such that
\begin{equation}\label{GN inequality}
 \|u\|_{L^q} \le C_{m,q} \|u\|_{\W}^\alpha \|u\|_{L^2}^{1-\alpha}, \ \ u\in\W.
\end{equation}

Now we define
$$  H^{m}:=\left\{ u\in\W: \  \text{div}  u=0     \right\}.$$
The norm of $\W$ restricted to $H^m$ will be denoted by $\|\cdot\|_{\H}$. We recall that
$H^0$ is a closed linear subspace of the Hilbert space $L^2(\L, \R^3)$.
 In the  literature it is well known that
one can  use the Gelfand triple $H^1 \subseteq H^0\subseteq (H^{1})^*$ to analyze the
Navier-Stokes equation and it works very well in the 2D case even with general stochastic perturbations
 (cf.\cite{BLZ,LR10,T84} and the references therein).
However, as pointed out in \cite{L11,LR10}, the growth condition $(H4)$
fails to  hold on this triple for the 3D Navier-Stokes equation.

Motivated by some recent papers on the (stochastic)  tamed 3D Navier-Stokes equation (cf. \cite{RZ10,RZZ,RZ09a,RZ09}),
we will use the following Gelfand triple in order to verify the growth condition $(H4)$:
 $$    V:=H^2 \subseteq H:=H^1\subseteq  V^* .  $$
The main reason is that we can use the following  inequality in the 3D case (see e.g.\cite{He01}):
\begin{equation}\label{Heywood}
 \sup_{x} |u(x)|^2 \le C \|\Delta u\|_{H^0} \|\nabla u\|_{H^0} .
\end{equation}
Let $\P$ be the orthogonal (Helmhotz-Leray) projection from $L^2(\L,\R^3)$ to $H^0$ (cf.\cite{T84,Li96}). It is well known that
$\P$ 
can be restricted to a bounded linear operator from
$\W$ to $\H$.  For any $u\in H^0$ and $v\in L^2(\L,\R^3)$ we have
$$ \<u, v\>_{H^0} :=  \<u, \P v\>_{H^0}=\<u, v\>_{L^2} . $$
Then by means of  the divergence free Hilbert
spaces $H^2, H^1$ and the  orthogonal projection $\P$,  the
classical  3D Navier-Stokes equation (\ref{3D NSE}) can be reformulated in
the following abstract form:
\begin{equation}\label{NSE}
u'=Au+B(u)+F,\ u(0)=u_0\in H^1,
\end{equation}
where
$$ A: H^2\rightarrow V^*,\    Au=\nu \P \Delta u;$$
  $$ B: H^2\times H^2\rightarrow V^*,   \  B(u,v)=-\P \left[(u \cdot \nabla) v\right],\  B(u)=B(u,u); $$
  $$ F: [0,T] \rightarrow H^0$$
are well defined.

\begin{rem} (1)  It is obvious that $H^0\subseteq L^2(\L,\R^3)\subseteq V^*$ and
$$    \|u\|_{V^*} \le \|u\|_{L^2}=\|u\|_{H^0},\ u\in H^0.     $$
(2)  It is well known that
$$ \<B(u,v),w\>_{L^2}=-\<B(u,w),v\>_{L^2}, \   \<B(u,v),v\>_{L^2}=0,\   u,v,w\in H^2. $$
However, one should note that
$$ \<B(u,v),v\>_{H^2}:=   { }_{H^0} \<B(u,v), v\>_{H^2}=\<B(u,v), (I-\Delta)v\>_{L^2},\   u,v,w\in H^2,  $$
which might  not be  equal to $0$ in general.

Therefore,  it is not obvious whether the usual coercivity condition
still holds on this new triple or not?   In fact,  this is  one reason that  we introduce  a generalized coercivity
condition in order to handle this nonlinear term using  this new triple.
\end{rem}

For simplicity we only apply Theorem \ref{T1} to the deterministic 3D Navier-Stokes equation.
But  one can also add a general type additive noise to (\ref{NSE}) and obtain
the corresponding result in the stochastic case by applying Theorem \ref{T1.1} and Remark \ref{r1}.

\begin{exa}(3D Navier-Stokes equation)
If $F\in L^2(0,T; H^0)$ and $u_0\in H^1$, then there exists a constant $T_0\in (0, T]$ such that
 $(\ref{NSE})$ has a unique strong solution $u\in L^2([0,T_0]; H^2)\cap C([0,T_0]; H^1)$.

In particular, it is enough to choose   $T_0\in(0,T]$ such that   the following property holds:
$$  T_0< \frac{C}{  \|u_0\|_{H^1}^2 +\int_0^{T_0} (1+\|F(t)\|_{L^2}^2) d t }\   , $$
where $C>0$ is  some  (given)  constant only depending on the viscosity constant $\nu$.
\end{exa}
\begin{proof} The hemicontinuity $(H1)$ is easy to verify  since $B$ is a bilinear map.

By (\ref{Heywood}) and Young's inequality   we have
\begin{equation}\label{e7}
 \begin{split}
&~  \<B(u)-B(v),u-v\>_V \\
=& \<B(u)-B(v), (I-\Delta)(u-v)\>_{L^2}  \\
 \le& \|u-v\|_{V}  \|(u\cdot \nabla)u - (v\cdot \nabla) v\|_{L^2}  \\
 \le&  \|u-v\|_{V}  \left( \|u\|_{L^\infty} \|\nabla u- \nabla v  \|_{L^2} +\|u-v\|_{L^\infty} \|\nabla v\|_{L^2} \right)  \\
 \le&  \|u-v\|_{V}  \left( \|u\|_{L^\infty} \| u-  v  \|_{H} +C\|u-v\|_{V}^{1/2} \|u-v\|_{H}^{1/2} \| v\|_{H} \right)  \\
 \le&  \frac{\nu}{2} \|u-v\|_V^{2} + C \left( \|u\|_{L^\infty}^2 +\|v\|_{H}^4 \right)   \|u-v\|_H^{2},
\   u, v\in V,
 \end{split}
 \end{equation}
where $C>0$ is a constant only depending on $\nu$.

Hence we have the following  local monotonicity $(H2)$:
\begin{equation*}
 \begin{split}
  &  \<Au+B(u)-Av-B(v),u-v\>_V \\
\le& -\frac{\nu}{2} \|u-v\|_V^{2} + \nu  \|u-v\|_H^{2}
  +C \left( \|u\|_{L^\infty}^2 +\|v\|_{H}^4 \right)   \|u-v\|_H^{2},\ u,v\in V.
 \end{split}
\end{equation*}
In particular, there exists a constant $C$ such that  (let $u=0$)
$$  \<Av+B(v), v\>_V
\le -\frac{\nu}{2} \|v\|_V^{2} +  C (1+\|v\|_{H}^6),\   v\in V. $$
Then  it is easy to show that  $(H3)$ holds with $g(x)=C x^3$:
\begin{equation*}
\begin{split}
  \<Av+B(v)+F, v\>_V&\le -\frac{\nu}{2} \|v\|_V^2+ C(1+ \|v\|_{H}^6 )  +\|F\|_{V^*}\|v\|_V \\
&\le
  -\frac{\nu}{4}\|v\|_V^2+ C  \|v\|_{H}^6 + C \left( 1 +    \|F\|_{L^2}^2 \right), \  v\in V.
\end{split}
\end{equation*}
Note that by (\ref{Heywood}) we have
\begin{equation}\label{e8}
  \|B(v)\|_{V^*}^2\le \|(v\cdot \nabla)v\|_{L^2}^2\le \|v\|_{L^\infty}^2 \|\nabla v\|_{L^2}^2
\le C \|v\|_V \|v\|_H^3 \le C \|v\|_V^2 \|v\|_H^2,  \ v\in V.
\end{equation}
Hence  $(H4)$ holds with $\beta=2$.

Then  the local  existence and uniqueness of solutions to (\ref{NSE}) follows from Theorem \ref{T1}.
\end{proof}

\begin{rem} Note that the solution here is a strong solution  in the sense of PDE.
  It is obvious that we can also allow $F$ in (\ref{NSE}) to depend on the unknown solution $u$ provided
 $F$ satisfies some
locally monotone condition (cf.\cite{L11}).
\end{rem}

\begin{rem} If we  analyze  (\ref{NSE}) by using the following Gelfand triple
$$  V:=H^2 \subseteq H:= H^0 \subseteq V^* ,  $$
 then  $ \<B(v), v\>_V=0$ and we have the classical coercivity  (i.e. $(H3)$ with $g(x)=Cx$):
$$ \<Av+B(v)+F, v\>_V\le -\nu\|v\|_V^2+\nu \|v\|_H^2+ \|F\|_{V^*}\|v\|_V \le
  -\frac{\nu}{2}\|v\|_V^2+\nu \|v\|_H^2+\frac{1}{2 \nu}\|F\|_{V^*}^2, \  v\in V.   $$
By (\ref{Heywood}) and Young's inequality we have
\begin{equation}
 \begin{split}
  \<B(u)-B(v),u-v\>_V &=-  \<B(u,u-v),v\>_V+  \<B(v,u-v),v\>_V \\
 &= -  \<B(u-v),v\>_V \\
 &\le   \|u-v\|_{L^\infty} \|\nabla(u-v)\|_{L^2}   \|v\|_{L^\infty } \\
&\le   \|u-v\|_{H}^\frac{1}{2} \|\nabla(u-v)\|_{L^2}^{\frac{3}{2}}   \|v\|_{L^\infty } \\
&\le   \|u-v\|_{H}^\frac{5}{4} \| u-v \|_{V}^{\frac{3}{4}}   \|v\|_{L^\infty } \\
& \le \frac{\nu}{2} \|u-v\|_V^{2} +  C  \|v\|_{L^\infty}^{\frac{8}{5}} \|u-v\|_H^{2},
\ u,v\in V.
 \end{split}
\end{equation}
Hence we have the local monotonicity $(H2)$:
$$  \<Au+B(u)-Av-B(v),u-v\>_V
\le -\frac{\nu}{2} \|u-v\|_V^{2} + C\left(1+  \|v\|_{L^\infty }^{\frac{8}{5}} \right) \|u-v\|_H^{2}.  $$
Concerning the growth condition   we have,
\begin{equation}
 \begin{split}
   \|B(u)\|_{V^*}^2 & \le  \|u\|_{L^\infty}^2 \|\nabla u\|_{L^2}^2 \\
 &  \le C \|u\|_V \|\nabla u\|_{L^2}^{3}  \\
& \le C\|u\|_V^{\frac{5}{2}} \|u\|_H^{\frac{3}{2}} ,  \ u\in V.
 \end{split}
\end{equation}
However,  this is  not enough to verify $(H4)$.
\end{rem}

\subsection{Tamed 3D Navier-Stokes equation}
In the case of  the 3D Navier-Stokes equation we see that the generalized coercivity condition holds with $g(x)=Cx^3$, hence we only
get  local existence and uniqueness of solutions. In this part  we consider  a tamed version of the (stochastic) 3D Navier-Stokes equation,
which was proposed recently in \cite{RZ09a,RZ09} (see also \cite{RZZ,RZ10}). The main feature of this tamed equation
is that if there is a bounded smooth solution to the classical 3D Navier-Stokes equation (\ref{3D NSE}), then this smooth solution must also
satisfy the following  tamed equation (\ref{Tamed NSE}) (for  $N$ large enough):
\begin{equation}\label{Tamed NSE}
\begin{split}
   &  \partial_t u(t)=\nu \Delta u(t)-(u(t)\cdot \nabla)u(t)+\nabla p(t) -g_N\left(|u(t)|^2\right) u(t)  +F(t), \\
   & \text{div} (u)=0,   ~ u(0)=u_0,\\
  &  u|_{\partial \Lambda}=0,
\end{split}
\end{equation}
where the taming function $g_N : \mathbb{R}_+\rightarrow
\mathbb{R}_+$ is smooth and satisfies for some $N > 0$,
$$\begin{cases} & g_N(r)=0,\  \text{if}\  r\le N,\\
 & g_N(r)=(r-N)/\nu,\   \text{if}\  r\ge N+1,\\
 & 0\le g_N^\prime(r)\le C,\  \  r\ge 0.
\end{cases}
$$

\begin{exa}(Tamed 3D Navier-Stokes equation)
For $F\in L^2(0,T; H^0)$ and $u_0\in H^1$,
 $(\ref{Tamed NSE})$ has a unique  strong  solution  $u\in L^2([0,T]; H^2)\cap C([0,T]; H^1)$.
\end{exa}
\begin{proof} Without loss of generality we may assume $\nu=1$ for simplicity.

 Using the Gelfand triple
 $$    V:=H^2 \subseteq H:=H^1\subseteq  V^*,$$
(\ref{Tamed NSE}) can be  rewritten  in
the abstract form:
$$   u'=Au+B(u)- \P\left[g_N\left(|u|^2\right) u \right] +F,\ u(0)=u_0\in H^1,  $$
We recall the following estimates  for $v\in H^2$ from the proof of \cite[Lemma 2.3]{RZ09}:
\begin{equation}
\begin{split}
  \<Av, v\>_V&=\<\P\Delta v, (I-\Delta) v\>_{L^2}\le -\|v\|_{V}^2+ \|v\|_{H}^2 ;\\
\<B(v), v\>_V&=-\<\P(v \cdot \nabla)v,  (I-\Delta)v\>_{L^2} \le \frac{1}{4} \|v\|_{V}^2+\frac{1}{2}\| |v|\cdot |\nabla v| \|_{L^2}^2 ;\\
-\<\P\left[g_N(|v|^2)v \right], v\>_V&=-\<\P\left[g_N(|v|^2)v \right], (I-\Delta)v\>_{L^2} \le   -\| |v|\cdot |\nabla v| \|_{L^2}^2+ CN \|v\|_{H}^2.
\end{split}
\end{equation}
Then  it is  easy to get the following coercivity $(H3)$ with $g(x)=C(N+1)x$:
$$ \<Av+B(v)-\P\left[g_N(|v|^2)v \right]+F, v\>_V\le -\frac{1}{2}  \|v\|_V^2 + C(N+1) \|v\|_{H}^2  + C\|F\|_{V^*}^2,  \  v\in V.   $$
By (\ref{Heywood})  we have
\begin{equation}
 \begin{split}
&~  -\<\P\left[g_N(|u|^2)u \right] -\P\left[g_N(|v|^2)v \right], u-v\>_V \\
=& -\<\P\left[g_N(|u|^2)u \right] -\P\left[g_N(|v|^2)v \right], (I-\Delta)(u-v)\>_{L^2}  \\
 \le& \|u-v\|_{V}  \|g_N(|u|^2)u -g_N(|v|^2)v\|_{L^2}  \\
 \le& \|u-v\|_{V}  \| \left(g_N(|u|^2)- g_N(|v|^2)   \right)u -g_N(|v|^2) \left(u- v \right)\|_{L^2}  \\
 \le&  \|u-v\|_{V}  \left(C \|u-v\|_{L^\infty} \| |u|^2+ | v|^2  \|_{L^2} +C\||v|^2\|_{L^2}\|u-v\|_{L^\infty} \right)  \\
 \le&  C \|u-v\|_{V}^{\frac{3}{2}} \|u-v\|_{H}^{\frac{1}{2}}  \left( \|u\|_{L^4}^2 + \|  v  \|_{L^4}^2  \right)  \\
 \le&  \frac{1}{4} \|u-v\|_V^{2} + C \left( \|u\|_{L^4}^8 +\|v\|_{L^4}^8 \right)   \|u-v\|_H^{2},
\   u, v\in V,
 \end{split}
 \end{equation}
where $C$ is constant changing from line to line.

Hence by (\ref{e7})  we have the following  estimate  (note that $\nu=1$):
\begin{equation}
 \begin{split}
&  \<Au+B(u)-\P\left[g_N(|u|^2)u \right] -Av-B(v)+\P\left[g_N(|v|^2)v \right], u-v\>_V \\
\le& -\frac{1}{4} \|u-v\|_V^{2} +
 C \left( 1+\|u\|_{L^\infty}^2+\|u\|_{L^4}^8 + \|v\|_{H}^4+ \|v\|_{L^4}^8 \right)   \|u-v\|_H^{2},\ u,v\in V,
 \end{split}
\end{equation}
i.e. $(H2)$ holds with $\rho(u)=\|u\|_{L^\infty}^2+\|u\|_{L^4}^8$ and $\eta(v)= \|v\|_{H}^4+ \|v\|_{L^4}^8$.

By (\ref{GN inequality}) we have
$$  \|\P\left[g_N(|v|^2)v \right]\|_{V^*}^2\le C \|v\|_{L^6}^2 \le C \|v\|_H^2,
 v\in V. $$
Then by (\ref{e8}) we obtain that   $(H4)$ holds with $\beta=2$.

Since (\ref{c3}) also holds,    the  global  existence and uniqueness of solutions to (\ref{Tamed NSE}) follows from Theorem \ref{T1}.
\end{proof}

\subsection{Cahn-Hilliard equation}
The Cahn-Hilliard equation is a classical model to describe phase separation in a binary alloy and some other media, we refer to \cite{N98} for a survey on this model
(see also \cite{EM91,DD96} for the stochastic case).
Let $\Lambda$ be a bounded open domain in $\R^d ~ (d\le 3)$ with smooth boundary. The  Cahn-Hilliard equation
has the following form:
\begin{equation}\label{Cahn-Hilliard}
\begin{split}
 &\partial_t u=-\Delta^2 u+ \Delta\varphi(u), \  u(0)=u_0, \\
& \frac{\partial}{\partial n} u= \frac{\partial}{\partial n} (\Delta u)=0   \  \  \text{on}\  \  \partial \Lambda,
\end{split}
\end{equation}
where $\Delta$ is the Laplace operator, $n$ is the outward unit normal vector on the boundary $\partial \Lambda$
and  the nonlinear term $\varphi$  is some polynomial function.

Now we consider the following Gelfand triple
$$ V \subseteq H:= L^2(\Lambda) \subseteq V^*,   $$
where $V:=\{ u\in W^{2,2} (\Lambda):  \  \frac{\partial}{\partial  n} u= \frac{\partial}{\partial  n} (\Delta u)=0
  \  \text{on} \ \partial \Lambda  \} $.

Then we get the following existence and uniqueness result for (\ref{Cahn-Hilliard}).

\begin{exa} Suppose that $\varphi\in C^1(\R)$ and  there exist  some positive constants $C$ and   $p\le \frac{d+4}{d}$ such that
\begin{align*}
& \varphi^\prime(x) \ge -C, \   |\varphi(x)|  \le C(1+|x|^p),  \  x\in\R; \\
& |\varphi(x)-\varphi(y) |\le C(1+|x|^{p-1}+|y|^{p-1}) |x-y|, \  x,y\in\R.
\end{align*}
Then for any $u_0\in L^2(\Lambda)$, there exists a unique solution to  $(\ref{Cahn-Hilliard})$.
\end{exa}
\begin{proof}
For any $u,v\in V$, we have
$$ - \<\Delta^2 u- \Delta^2 v , u-v\>_V=-\|u-v\|_V^2. $$
By  the assumptions on $\varphi$ and  Young's inequality we get
\begin{align*}
 &  \<\Delta \varphi(u)- \Delta \varphi(v) , u-v\>_V \\
\le &   \|u-v\|_{V} \|\varphi(u)-\varphi(v)\|_{L^2}  \\
\le &  \|u-v\|_V  \cdot  C\left(   1+ \|u\|_{L^\infty}^{p-1} +\|v\|_{L^\infty}^{p-1}    \right) \|u-v\|_{L^2}     \\
\le& \frac{1}{2} \|u-v\|_V^2 + C\left(   1+ \|u\|_{L^\infty}^{2p-2} +\|v\|_{L^\infty}^{2p-2}    \right) \|u-v\|_{H}^2,
\  u,v\in V.
\end{align*}
Hence $(H2)$ holds with $\rho(u)=\eta(u)=C\|u\|_{L^\infty}^{2p-2}$.

Similarly, by the interpolation inequality we have  for any $v\in V$,
\begin{align*}
 \<\Delta \varphi(v) , v\>_V
=- \int_\Lambda \varphi^\prime(v) |\nabla v|^2 d x
\le  C\|v\|_{W^{1,2}}^2
\le \frac{1}{2} \|v\|_V^2 + C\|v\|_{H}^2,
\end{align*}
i.e. $(H3)$ holds with $\alpha=2$ and $g(x)=Cx$.

 It is also easy to see that
\begin{align*}
  \|\Delta \varphi(v)\|_{V^*}  &\le  \|\varphi(v) \|_{H} \\
 & \le C\left( 1+\|v\|_{L^{2p}}^p \right) \\
& \le C \left(1+ \|v\|_V^{\frac{(p-1)d}{4}}   \|v\|_H^{\frac{d+(4-d)p}{4}}  \right) \\
& \le C\left(1+ \|v\|_V^{\frac{(p-1)d}{4}}   \right ) \left(1+  \|v\|_H^{\frac{d+(4-d)p}{4}}  \right),
\ v\in V.
\end{align*}
Since  $p\le \frac{4}{d}+1$ (i.e. $\frac{(p-1)d}{4}\le 1$) and $\|v\|_H\le C \|v\|_V$,  we have
$$   \|\Delta \varphi(v)\|_{V^*}\le C \left(1+ \|v\|_V  \right ) \left(1+  \|v\|_H^{p-1}  \right),
\ v\in V, $$
i.e. $(H4)$  holds with $\beta=p-1$.

Note that for any $v\in V$,
$$ \rho(v)= C\|v\|_{L^\infty}^{2p-2} \le C \|v\|_{V}^{\frac{(p-1)d}{2}}  \|v\|_H^{\frac{(p-1)(4-d)}{2}},  $$
i.e. (\ref{c3}) also holds.

Therefore, the  conclusion follows directly from Theorem \ref{T1}.
\end{proof}

\subsection{Surface growth PDE with random noise}
We consider a model which appears in the theory of growth of surfaces,
which describes an amorphous material deposited on an initially flat surface in high vacuum (cf.\cite{RLH,BFR09} and the references therein).
 Taking account of random noises   the equation is formulated on the interval $[0, L]$ as follows:
\begin{equation}\label{surface growth PDE}
d X(t)=\left[ -\partial_x^4 X(t)-\partial_x^2 X(t)+\partial_x^2(\partial_x X(t))^2 \right] d t +B(t) d W(t), \  X(0)=x_0,
\end{equation}
where $\partial_x, \partial_x^2, \partial_x^4$ denote the first, second and fourth spatial derivatives respectively.

Recall that $W(t)$ is a $U$-valued cylindrical Wiener process.
Using the following Gelfand triple
$$  V:=W_0^{4,2}([0,L]) \subseteq H:= W^{2,2}([0,L]) \subseteq V^*   $$
we can obtain the following local existence and uniqueness of strong solutions for (\ref{surface growth PDE}).

\begin{exa} Suppose that $B\in L^2([0,T]; L_2(U; H))$.
For any $X_0\in L^2(\Omega\rightarrow H; \mathcal{F}_0; \mathbb{P})$, there exists a unique local solution $\{X(t)\}_{t\in[0,\tau]}$ to (\ref{surface growth PDE})  satisfying
$$  X(\cdot)\in L^2([0,\tau]; V)\cap C([0,\tau]; H^2), \mathbb{P}\text{-a.s.}. $$
\end{exa}
\begin{proof} It is sufficient to verify $(H1)$-$(H4)$ for (\ref{surface growth PDE}), then the conclusion follows from Theorem \ref{T1.1}.


For $u,v\in V$,  by  standard interpolation inequalities and Young's inequality
 we have
\begin{align*}
 &  \<\partial_x^2(\partial_x u)^2- \partial_x^2(\partial_x v)^2, u-v\>_V \\
=& \<\partial_x^2(\partial_x u)^2- \partial_x^2(\partial_x v)^2, \partial_x^4 u-\partial_x^4 v\>_{L^2} \\
\le & \|u-v\|_V  \| \partial_x^2(\partial_x u)^2- \partial_x^2(\partial_x v)^2\|_{L^2} \\
\le & \|u-v\|_V \left[ \| (\partial_x^2 u)^2- (\partial_x^2 v)^2\|_{L^2}+
                 \| \partial_x u \partial_x^3 u - \partial_x v \partial_x^3 v \|_{L^2}     \right] \\
\le & \|u-v\|_V \left[ (\|\partial_x^2 u\|_{L^\infty} +\|\partial_x^2 v\|_{L^\infty}) \|u-v\|_H
    + \|\partial_x u\|_{L^\infty} \|\partial_x^3 u- \partial_x^3 v\|_{L^2}
    +\|\partial_x^3 v\|_{L^2} \|\partial_x u- \partial_x v\|_{L^\infty}   \right] \\
  \le   & \|u-v\|_V \left[ (\|\partial_x^2 u\|_{L^\infty} +\|\partial_x^2 v\|_{L^\infty}) \|u-v\|_H
    + \|\partial_x u\|_{L^\infty} \| u- v\|_{V}^{\frac{1}{2}}  \| u- v\|_{H}^{\frac{1}{2}}
    +\|\partial_x^3 v\|_{L^2} \|u- v\|_{H}   \right] \\
   \le & \frac{1}{4}\|u-v\|_V^2 +C\left(   \| u\|_{W^{2,\infty}}^2+\| u\|_{W^{1,\infty}}^4 +
   \| v\|_{W^{2,\infty}}^2 +\| v\|_{W^{3,2}}^2 \right) \|u-v\|_H^2,
\end{align*}
where $C$ is some constant.

Note that
\begin{align*}
 &  \<-\partial_x^4 u- \partial_x^2u + \partial_x^4 v +\partial_x^2 v, u-v\>_V \\
\le& -\|u-v\|_V^2 + \|u-v\|_V \|u-v\|_H \\
\le& -\frac{3}{4} \|u-v\|_V^2 + \|u-v\|_H^2.
\end{align*}
Hence we know that $(H2)$ holds with
$$  \rho(u)=  \| u\|_{W^{2,\infty}}^2+\| u\|_{W^{1,\infty}}^4, \
\eta(v)=\| v\|_{W^{2,\infty}}^2 +\| v\|_{W^{3,2}}^2.   $$
Similarly,
\begin{align*}
 \|\partial_x^2(\partial_x v)^2\|_{V^*} \le & \| (\partial_x^2 v)^2 + \partial_x v \partial_x^3 v\|_{L^2} \\
\le & \|v\|_{W^{2,4}}^2 + \|v\|_{W^{1,\infty}} \|v\|_{W^{3,2}} \\
\le & C \|v\|_V^{\frac{1}{2}} \|v\|_{H}^{\frac{3}{2}}, \  v\in V,
\end{align*}
i.e. $(H4)$ holds with $\beta=1$.

Moreover, this also implies that
$$2 \<\partial_x^2(\partial_x v)^2, v\>_V \le 2 \|v\|_V  \|\partial_x^2(\partial_x v)^2\|_{V^*}  \le C\|v\|_V^{\frac{3}{2}}\|v\|_H^{\frac{3}{2}}
\le\frac{1}{2}\|v\|_V^2 + C\|v\|_H^6.
 $$
 Since
 $$2 \<-\partial_x^4 v -\partial_x^2 v, v\>_V \le
\frac{3}{2}\|v\|_V^2 + \|v\|_H^2,
 $$
we deduce that $(H3)$ holds with $g(x)=C x^3$.

Now the proof is complete.
\end{proof}

\begin{rem} (1) It is  known in the literature  that the (1-dimension) surface growth model has some similar features of  difficulty
as the 3D Navier-Stokes equation,  the uniqueness of weak solutions for this model  is still an open problem  in both the  deterministic  and stochastic case.
From the proof above one can see these similarities (e.g.$(H2)$-$(H4)$)  very clearly  between this model and the  3D Navier-Stokes equation (Example 3.1).

(2) The solution obtained here for the stochastic surface growth model  is a strong solution  in the sense of both PDE and SPDE.
We should remark that for the space time white noise case, the existence of a weak martingale solution was obtained by Bl\"{o}mker,
Flandoli and Romito  in \cite{BFR09} for this model, and the existence of a Markov selection and ergodicity properties
were also proved there.
\end{rem}

\subsection{Stochastic power law fluids}
The next example of (S)PDE is a model which describes the velocity field of a viscous and
 incompressible non-Newtonian fluid subject to some random forcing.
The deterministic model has been studied intensively in PDE theory (cf.\cite{FR,MNRR} and the references therein).
Let $\Lambda$ be a  bounded domain in $\mathbb{R}^d (d\ge 2)$  with sufficiently smooth boundary.
For a vector field  (e.g. the velocity field  of the fluid)  $u: \Lambda\rightarrow \R^d$,
we denote the rate of strain tensor by
$$ e(u): \Lambda\rightarrow \R^d\otimes \R^d; \   e_{i,j} (u)=\frac{\partial_i u_j+ \partial_j u_i}{2},
\ i,j=1,\cdots, d.  $$
In this paper we consider the case  that the extra stress tensor has the following polynomial form:
$$ \tau(u): \Lambda \rightarrow\R^d\otimes\R^d; \ \tau(u)=2\nu (1+|e(u)|)^{p-2} e(u),   $$
where   $\nu>0$ is the kinematic viscosity and $p>1$ is some constant.

In the case of deterministic forcing,  the dynamics of  power law fluids can be modeled   by the following PDE:
\begin{equation}\label{power law fluids}
 \begin{split}
& \partial_t u= \text{div} \left(\tau(u) \right)- (u\cdot \nabla)u-\nabla p +f , \\
&  \text{div}  (u)=0,
 \end{split}
\end{equation}
where $u=u(t,x)=\left( u_i(t,x) \right)_{i=1}^d$ is the velocity field, $p$ is the pressure,
$f$ is some external force and
$$ u\cdot \nabla=\sum_{j=1}^d u_j \partial_j, \  \
 \text{div}\left( \tau (u) \right)= \left( \sum_{j=1}^d \partial_j \tau_{i,j} (u)  \right)_{i=1}^d.  $$
\begin{rem}
(1) Note that $p=2$ describes the Newtonian fluids and (\ref{power law fluids}) reduces to  the classical
Navier-Stokes equation (\ref{3D NSE}).

(2) The shear shining fluids (i.e. $p\in(1,2)$) and the shear thickening fluids (i.e.  $p\in (2, \infty)$) has been also widely
studied in different fields of science and engineering (cf.\cite{FR,MNRR}).
\end{rem}

Now we consider the following Gelfand triple
$$ V\subseteq H \subseteq V^*,   $$
where
$$ V=\left\{ u\in W_0^{1,p}(\Lambda; \R^d):\   \text{div}  u=0    \right\};
\  H=\left\{ u\in L^2(\Lambda; \R^d):\   \text{div}  u=0    \right\}.  $$
Let $\mathcal{P}$ be the orthogonal (Helmhotz-Leray) projection from
 $L^2(\L,\R^d)$ to $H$.
It is well known that  the following operators
$$  A: W_0^{2,p}(\Lambda; \R^d)\cap V \rightarrow H, \ A(u)= \mathcal{P} \left[  \text{div}( \tau( u ) ) \right] ; $$
$$        B: W_0^{2,p}(\Lambda; \R^d)\cap V \times W_0^{2,p}(\Lambda; \R^d)\cap V  \rightarrow H; \  B(u, v)=- \mathcal{P} \left[ (u\cdot \nabla) v   \right],  \ B(u):=B(u,u)  $$
can be extended to the   well defined operators:
$$     A: V\rightarrow V^*; \ B: V\times V\rightarrow V^*.              $$
In particular, one can show that
$$     \<A(u),  v\>_V = -  \int_\Lambda \sum_{i,j=1}^d  \tau_{i,j}(u) e_{i,j}(v) \d x;\ u,v\in V;     $$
$$       \<B(u,v),  w\>_V=  -   \<B(u,w),  v\>_V, \   \<B(u,v),  v\>_V=0, \ u,v,w\in V.         $$
Now (\ref{power law fluids}) can be reformulated in the following variational form:
\begin{equation}\label{PLF}
u^\prime (t)=  A(u(t))+B(u(t))+ F(t) , \  u(0)=u_0.
 \end{equation}
\begin{exa}
 Suppose that $u_0\in H, F\in L^2([0,T]; V^*)$ and $p\ge \frac{3d+2}{d+2}$, then
$(\ref{PLF})$ has a solution. Moreover, if $p\ge \frac{d+2}{2}$, then the solution of
$(\ref{PLF})$ is also unique.
\end{exa}

\begin{proof} Without loss of generality we may assume $\nu=1$.

We first recall the well known Korn's inequality for $p\in (1, \infty)$:
$$  \int_\Lambda |e(u)|^p \d x \ge C_p \|u\|_{1,p}, \  u\in W_0^{1,p}(\Lambda; \R^d),  $$
where  $C_p>0$ is  some constant.

The following inequalities are also used very often in the study of power law fluids (cf. \cite[pp.198 Lemma 1.19]{MNRR}):
\begin{equation}\label{e11}
 \begin{split}
&  \sum_{i,j=1}^d \tau_{i,j}(u) e_{i,j} (u)\ge C(|e(u)|^p-1); \\
&   \sum_{i,j=1}^d  (\tau_{i,j} (u)-\tau_{i,j} (v))(e_{i,j}(u)- e_{i,j}(v))\ge C\left(|e(u)-e(v)|^2 +|e(u)-e(v)|^p  \right);\\
& |\tau_{i,j} (u) | \le C(1+|e(u)|)^{p-1}, \ i,j=1, \cdots, d.
\end{split}
\end{equation}
Then by the interpolation inequality and Young's inequality  one can show  that
\begin{equation*}
 \begin{split}
 & ~~ \<B(u)-B(v),u-v\>_V\\
 &= -  \<B(u-v), v\>_V \\
&=  \<B(u-v,v), u-v\>_V \\
&\le C   \|v\|_{V}  \|u-v\|_{\frac{2p}{p-1}}^2  \\
 &\le C  \|v\|_{V}  \|u-v\|_{1,2}^{\frac{d}{p}} \|u-v\|_H^{\frac{2p-d}{p}}  \\
& \le \varepsilon \|u-v\|_{1,2}^{2} + C_\varepsilon \|v\|_{V}^{\frac{2p}{2p-d}} \|u-v\|_H^{2},
\ u,v\in V.
 \end{split}
\end{equation*}
By (\ref{e11}) and Korn's inequality  we have
\begin{equation*}
 \begin{split}
 &  \<A(u)-A(v), u-v\>_V\\
=&- \int_\Lambda \sum_{i,j=1}^d
\left( \tau_{i,j}(u) -\tau_{i,j}(v)\right) \left( e_{i,j}(u)-e_{i,j}(v) \right) \d x\\
\le& -C\|e(u)-e(v)\|_H^2\\
\le& -C\|u-v\|_{1,2}^2.
 \end{split}
\end{equation*}
Hence we have the following estimate:
$$  \<Au+B(u)-Av-B(v), u-v\>_V
 \le -(C-\varepsilon) \|u-v\|_{1,2}^{2} + C_\varepsilon \|v\|_{V}^{\frac{2p}{2p-d}} \|u-v\|_H^2, $$
i.e.  $(H2)$ holds with $\rho(v)=C_\varepsilon \|v\|_{V}^{\frac{2p}{2p-d}}$.

It is also easy to  verify $(H3)$ with $\alpha=p$ as follows:
$$  \<A(v)+B(v),  v\>_V\le -C_1\int_\Lambda |e(v)|^p \d x+ C_2\le -C_3\|v\|_V^p+C_2.   $$

Note that
$$
 \left| \<B(v),u\>_V \right|
= \left|  \<B(v,u), v\>_V \right| \le   \|u\|_{V}  \|v\|_{\frac{2p}{p-1}}^2,
\ u,v\in V,
 $$
hence we have
$$ \|B(v)\|_{V^*}\le   \|v\|_{\frac{2p}{p-1}}^2, \ v\in V.$$
Then by the interpolation inequality and Sobolev's inequality we have
$$    \|v\|_{\frac{2p}{p-1}}\le \|v\|_q^\gamma \|v\|_2^{1-\gamma} \le C \|v\|_V^\gamma\|v\|_H^{1-\gamma},    $$
where $q=\frac{dp}{d-p}$ and $\gamma=\frac{d}{(d+2)p-2d}$.

Note that $2\gamma\le p-1$ if $p\ge \frac{3d+2}{d+2}$, and it is also easy to see that
$$ \|A(v)\|_{V^*} \le C (1+\|v\|_V^{p-1}), \ v\in V. $$
 Hence  the growth condition $(H4)$ also holds.

Then  the   existence  of solutions to (\ref{PLF}) follows from Theorem \ref{T1}.
Moreover, if $d\ge \frac{2+d}{2}$, then (\ref{c3})  holds and hence the solution of (\ref{PLF})
is  unique.
\end{proof}

Now we consider the power law fluids with state-dependent random forcing which can be described by the following SPDE:
\begin{equation}\label{stochastic PLF}
 \d X(t)= \left( A(X(t))+B(X(t)) \right) \d t + Q(X(t)) \d W(t), \  X(0)=X_0,
\end{equation}
where $W(t)$ is a cylindrical Wiener process on a Hilbert space $U$ w.r.t. the filtered probability space $(\Omega,\F,  \F_t, \P)$.

The following result can be obtained similarly as in the previous example  using  \cite[Theorem 1.1]{LR10}.
\begin{exa}\label{E6}
Suppose that $p\ge \frac{2+d}{2}$,  $X_0\in L^4(\Omega\rightarrow H; \mathcal{F}_0, \mathbb{P})$ and $Q$ is a
map from $V$ to $L_{2}(U; H)$   such that
\begin{equation}\begin{split}
 \|Q(u)-Q(v)\|_{2}^2
     \le C \|u-v\|_H^2, \
u, v\in V.
\end{split}
\end{equation}
Then
 $(\ref{stochastic PLF})$ has a unique  strong  solution
 $X\in L^4([0,T]\times\Omega, \d t\times\mathbb{P}, V)\cap L^4(\Omega, \mathbb{P}, C([0,T]; H))$.
\end{exa}
\begin{rem} In \cite{TY11} the authors established the existence and uniqueness of weak solutions for (\ref{PLF})
with additive Wiener noise. They first  considered  the Galerkin approximation and  showed the tightness of the distributions
 of the corresponding approximating solutions.
Then they proved that the limit is a weak solution of (\ref{PLF}) with additive Wiener noise.

Here we apply the main result (Theorem \ref{T1})  directly  to (\ref{PLF})
and establish the existence and uniqueness of  solutions. Therefore, by Theorem \ref{T1.1} and Remark \ref{r1} we
can obtain the existence and uniqueness of strong solutions (in the sense of SPDE) for (\ref{PLF}) with general additive type noises.
Moreover, as just showed in Example \ref{E6},  we can also even prove the analogous result for  (\ref{PLF})  with  multiplicative Wiener noise.
 \end{rem}

 \section*{Acknowledgements}
 The authors would like to thank Wilhelm Stannat for drawing our attention to the stochastic
 surface growth model, and also thank  Rongchan Zhu and Xiangchan Zhu for their helpful discussions.
 This work was supported by the DFG through
the SFB-701, and the IGK ``Stochastic and Real World Models" (IRTG 1132) as well as the BiBoS Research Center.

\bibliographystyle{amsplain}



\end{document}